\documentclass[12pt]{amsart}

\usepackage{verbatim, amssymb, enumitem} 

\usepackage[breaklinks=true,colorlinks=true,linkcolor=blue,citecolor=red,urlcolor=blue,psdextra,pdfencoding=auto]{hyperref} %backref,

\numberwithin{equation}{section}

\renewcommand \a{\alpha}
\renewcommand \b{\beta}
\newcommand \K{\delta}

\newcommand \la{\lambda}

\newcommand \id{\mathrm{id}}
\newcommand \br{\mathbb{R}}
\newcommand \bc{\mathbb{C}}
\newcommand \bh{\mathbb{H}}

\newcommand \rk{\operatorname{rk}}
\newcommand \Ker{\operatorname{Ker}}

\newcommand \End{\operatorname{End}}

\renewcommand \Im{\operatorname{Im}}

\newcommand \Span{\operatorname{Span}}
\newcommand \Tr{\operatorname{Tr}}

\newcommand \SO{\mathrm{SO}}
\newcommand \SU{\mathrm{SU}}
\newcommand \Sp{\mathrm{Sp}}

\newcommand \sC{\mathsf{C}}
\newcommand \sP{\mathsf{P}}
\newcommand \sN{\mathsf{N}}

\newcommand \cN{\mathcal{N}}

\newcommand \cC{\mathcal{C}}

\newcommand \cP{\mathcal{P}}

\newcommand \cV{{V}}

\newcommand \cW{{V}}

\newcommand\ag{\mathfrak a}
\newcommand\kg{\mathfrak k}
\newcommand\g{\mathfrak g}
\newcommand\cs{\mathfrak c}
\newcommand\h{\mathfrak h}
\newcommand\z{\mathfrak z}
\newcommand\m{\mathfrak m}
\newcommand \so{\mathfrak{so}}
\newcommand \spg{\mathfrak{sp}}
\newcommand \gl{\mathfrak{gl}}
\newcommand \ug{\mathfrak{u}}
\newcommand \su{\mathfrak{su}}
\newcommand \s{\mathfrak{s}}
\newcommand \n{\mathfrak{n}}
\newcommand \f{\mathfrak{f}}

\renewcommand\t{\mathfrak t}

\newcommand \Ng{\mathfrak{N}}
\newcommand \Pg{\mathfrak{P}}

\newcommand \ad{\operatorname{ad}}

\newcommand \diag{\operatorname{diag}}

\newcommand \<{\langle}
\renewcommand \>{\rangle}
\newcommand \ip{{\<\cdot,\cdot\>}}
\newcommand \ipr{{\<\cdot,\cdot\>}}

\newtheorem{theorem}{Theorem}
\newtheorem{theorema}{Theorem}

\newtheorem*{theorem*}{Theorem}
\newtheorem{corollary}{Corollary}
\newtheorem*{corollary*}{Corollary}
\newtheorem*{conj*}{Conjecture}
\newtheorem{lemma}{Lemma}
\newtheorem{proposition}{Proposition}
\newtheorem*{prop*}{Proposition}

\theoremstyle{definition}

\newtheorem*{definition*}{Definition}

\theoremstyle{remark}
\newtheorem{remark}{Remark}

\newtheorem*{notation*}{Notation}
\newtheorem*{algorithm*}{Algorithm}
\newtheorem*{example*}{Example}
\newtheorem*{observation*}{Observation}

\begin{document}

\title{Non-singular geodesic orbit nilmanifolds}

\author{Y.~Nikolayevsky}
\address{Department of Mathematical and Physical Sciences, La Trobe University, Melbourne, Australia 3086}
\email{y.nikolayevsky@latrobe.edu.au}
\thanks{The first named author was partially supported by ARC Discovery Grant DP210100951. The first named author is thankful to the University of Pennsylvania for hospitality.}

\author{W.~Ziller}
\address{Department of Mathematics, University of Pennsylvania, Philadelphia, PA 19104, USA}
\email{wziller@math.upenn.edu}
\thanks{The second named author was supported by an ROG grant from the University of Pennsylvania.}

% 53C30  (1973-now) Homogeneous manifolds
% 53C25  (1973-now) Special Riemannian manifolds
% 17B30  Solvable, nilpotent (super)algebras
% 22E25  Nilpotent and solvable Lie groups

%\keywords{non-singular nilmanifold, geodesic orbit manifold}

\begin{abstract}
A Riemannian manifold is called a \emph{geodesic orbit} manifolds, GO for short, if any geodesic is an orbit of a one-parameter group of isometries.  By a result of C.Gordon, a non-flat GO nilmanifold  is necessarily a two-step nilpotent Lie group with a left-invariant metric.  We give a complete classification of non-singular GO nilmanifolds. Besides previously known examples, there are new families  with 3-dimensional center, and two one-parameter families of dimensions 14 and 15.
\end{abstract}

\maketitle

\section{Introduction}
\label{s:intro}

As introduced in \cite{KV}, a Riemannian manifold $(M,g)$ is called a \emph{geodesic orbit manifold}, or a GO manifold for short, if any geodesic of $M$ is an orbit of a one-parameter subgroup of the full isometry group of $g$.   Any  geodesic orbit manifold is clearly homogeneous. The class of geodesic orbit manifolds includes (but is not limited to) symmetric, weekly symmetric  and naturally reductive spaces, as well as  generalized normal homogeneous spaces. For an up-to-date account of the state of knowledge on geodesic orbit manifolds we refer the reader to \cite{BN} and the bibliographies therein, as well as the more recent  \cite{Nik}.

By a result of \cite[Theorem~1.14]{Gor}, the study of general geodesic orbit manifolds can be, to an extent, reduced to the study of such in the following three cases: $M$ is a nilmanifold, $M$ is compact, or $M$ admits a transitive group of isometries which is semisimple of noncompact type.

In this paper, we study geodesic orbit nilmanifolds and give a full classification of \emph{non-singular} geodesic orbit nilmanifolds. Our starting point are the following two results of \cite{Gor}.

\begin{theorema}[{\cite[Theorem~2.2]{Gor}}] \label{tha:gor2}
  Any geodesic orbit nilmanifold is either two-step or abelian.
\end{theorema}
To state the second result we recall the following standard construction.
Let $(M,g)$ be a metric $2$-step nilpotent Lie group and $(\n, \ip)$ the corresponding metric $2$-step nilpotent Lie algebra. Let $\z = [\n,\n]$,  $\ag=\n^\perp$ and
\begin{equation}\label{jz}
	J_Z \in \so(\ag) \text{ with } \<J_ZX,Y\>=\<Z,[X,Y]\>, \text{ for all } X,Y \in \ag \text{ and } Z\in \z.
\end{equation}
This defines a linear subspace
\begin{equation*}
	\cV=\Span\{  J_Z\mid Z\in \z\} \subset \so(\ag).
\end{equation*}
 Since the linear map $ Z \mapsto J_Z$ is injective, we have  $\dim \cV =\dim \z$ and we may use both interchangeably. Moreover, $\cV$ inherits an inner product, which for simplicity we again denote by $\ip$, from the push-forward of the inner product on $\z\subset\n$.

Conversely, given a Euclidean vector space $\ag$ and a  linear subspace     $\cV\subset \so(\ag)$ with inner product $\ip$,  one  defines a metric nilpotent Lie algebra  $\n$ by setting $\n=\cV\oplus\ag$ with  $\cV=[\n,\n]$  and Lie brackets $[\ag,\ag]\subset \cV$ given by \eqref{jz}. Furthermore, it inherits an inner product  by declaring the decomposition to be orthogonal and using the given inner product on $\cV$ and $\ag$. After fixing a Euclidean vector space $\ag$, this gives a one-to-one  correspondence between pairs $(\cW, \ip)$ with $\cW\subset \so(\ag)$ and simply connected metric 2-step nilpotent Lie groups.

Given such  a pair $(\cW, \ip)$, we define the skew-symmetric normalizer subalgebra $\Ng \subset \so(\ag)$  by:
\begin{equation*}
	 \Ng =\{ N \in \so(\ag) \mid N (\cW ) N^{-1}\subset \cW, \text{ and }  (\ad_N)_{|\cW}: \cW \to \cW \text{ is skew-symmetric}  \}
\end{equation*}  relative to the inner product $\ip$ on $\cW$.
\begin{definition*}
	A pair $(\cW , \ipr)$ with $\cW\subset \so(\ag)$ is called a \emph{GO-pair} if for every $J \in \cW$ and every $X \in \ag$ there exists
	$$N = N(J,X) \in \Ng\ \text{ such that }\ [N,J]=0 \text{ and }\ NX=JX.$$
\end{definition*}

We then have:
\begin{theorema}[{\cite[Theorem~2.10]{Gor}}] \label{tha:gorcond}
Let $(M,g)$ be a simply connected metric $2$-step nilpotent Lie group and $(\n, \ip)$ the corresponding metric $2$-step nilpotent Lie algebra. Then $(M,g)$ is a GO manifold if and only if, in the above notation, $(\cV, \ipr)$ is a GO-pair.
\end{theorema}

Theorems~\ref{tha:gor2} and~\ref{tha:gorcond} reduce the classification of GO nilmanifolds to the classification of GO-pairs. We have the following special cases of GO-pairs:
 \smallskip

\begin{itemize}
	\item $(\cV, \ipr)$ is called of \emph{Rep type} if $V$ is a subalgebra of $\so(\ag)$.
	\item  $(\cV, \ipr)$ is called of \emph{Clifford type} (also called generalized H-type), if $V$ is the linear span of $\dim V$ anticommuting complex structures on $\ag$.
	\item 	  $(\cV, \ipr)$ is called of \emph{centralizer type},  if in the above Definition of a GO-pair we replace the normalizer $ \Ng$ by the centralizer of $V$.
\end{itemize}
In the first case we have $V\subset\Ng $ with $N(J,X)=J$, and hence  $(\cV, \ipr)$
ia a GO-pair. The GO-pairs of Clifford type were classified in \cite{Rie,Lau}. In particular, the possible dimensions are given by $\dim\z=1,2$ or  $(\dim\ag,\dim\z)=(4,\{2,3\})$,  $(8,\{5,6,7\})$, $(16,7)$ or $ (24,7)$. GO-pairs of centralizer type were studied in \cite{Nik}.

A $2$-step nilpotent Lie algebra is called \emph{non-singular} if all nonzero elements of $\cV$ are invertible. We will also say that the subspace $\cV \subset \so(\ag)$ is non-singular. The class of such Lie algebras is surprisingly large, in particular not classified, unless $\dim V$ is small. For the current state of knowledge on non-singular $2$-step nilpotent Lie algebras, we refer the reader to~\cite{Ebe, LO} and the bibliographies therein.

Our main result is the classification of $2$-step nilpotent non-singular GO-pairs. In the theorem below, all direct sums are orthogonal, $\bh$ is the algebra of quaternions with the standard inner product, $\{\mathrm{i}, \mathrm{j} , \mathrm{k}\}$ is the standard basis for $\Im \bh$, and $L_a: \bh \to \bh$ is the left multiplication by $a \in \bh$. We denote $J_i, \; i=1, \dots, 7$ pairwise anticommuting, orthogonal complex structures on $\br^8$.

\begin{theorem} \label{th:class}
  Let $(V, \ip)$ be a $2$-step nilpotent non-singular GO-pair. Then  it is either of Clifford type, or belongs to one of  the following:
  \begin{enumerate}[label=\emph{(\alph*)},ref=\alph*]
    \item \label{it:th1} $\dim \z =1$. Then $V$ is spanned by a non-singular skew-symmetric matrix.

    \item \label{it:th2} $\dim \z =2$. We  identify $\ag$ with $\bh^p$ and then $V=\Span(J_1, J_2)$ with
    \begin{equation*}
    	J_1=(L_{a_1}, L_{a_2}, \dots, L_{a_p}),\quad  J_2=(L_{b_1}, L_{b_2}, \dots, L_{b_p}),
    \end{equation*}
    where $a_s, b_s \in \Im \bh$ are linearly independent.
    \item \label{it:th3}
    $\dim \z =3$. Two cases are possible:
    \begin{enumerate}[label=\emph{(\roman*)},ref=\roman*]
      \item \label{it:th3cen}
      Identify $\ag$ with $\bh^p$. Then  $V=\Span(J_1, J_2, J_3)$ with
      \begin{equation*}
     	\hspace{30pt}  J_1=(L_{a_1}, L_{a_2}, \dots, L_{a_p}),\quad  J_2=(L_{b_1}, L_{b_2}, \dots, L_{b_p}), \quad J_3=(L_{c_1}, L_{c_2}, \dots, L_{c_p}),
      \end{equation*}
      where $\{a_s,b_s,c_s\}$ is a basis of $\Im\bh$, for every $s = 1, \dots p$.
      \item \label{it:th3rep}
      Let $\ag= \bh^p \oplus W, \, p \ge 0$, and $\rho\colon\so(3)\to \so(W)$ be a quaternionic representation on $W=\br^{4q}, \, q>1$, with no $4$-dimensional subrepresentations. Then $V=\Span(J_1, J_2, J_3)$ with
      \begin{gather*}
      	\hspace{30pt}  J_1=(\la_1 L_{\mathrm{i}}, \dots, \la_p L_{\mathrm{i}}, \rho(\mathrm{i})),\quad  J_2 = (\la_1 L_{\mathrm{j}}, \dots, \la_p L_{\mathrm{j}},\rho(\mathrm{j}) ), \\
      \hspace{30pt}  J_3=(\la_1 L_{\mathrm{k}},  \dots, \la_p L_{\mathrm{k}},\rho(\mathrm{k})),
      \end{gather*}
     where $\la_s \ne 0$ and where we identify $\so(3)$ with $\Im\bh$.
    \end{enumerate}

    \item \label{it:th6}
    $\dim \z =6$. Then $\dim \ag=8$ and $V=\Span(J_1, J_2, J_3, J_4, J_5,J')$, where $J'=J_7 \cos \theta + J_6J_7 \sin \theta, \; \theta \in (0, \pi/2)$.

    \item \label{it:th7}
    $\dim \z =7$. Then
   $\dim \ag=8$ and $V=\Span(J_1, J_2, J_3, J_4, J_5,J_6,J')$, where $J'=J_7 \cos \theta + J_6J_7 \sin \theta, \; \theta \in (0, \pi/2)$.
  \end{enumerate}
\end{theorem}
The GO-pairs with $\dim \z=1, 2$, or $\dim \z=3$ in case~\eqref{it:th3cen} are precisely those of centralizer type, and if $\dim \z=3$ with $\la_s = 1$ for all $s=1, \dots, p$ in case~\eqref{it:th3rep}, they are of Rep type. The examples in~\eqref{it:th6} and~\eqref{it:th7} can be viewed as deformations of the Clifford GO-pair with $\theta=0$. For different values of $\theta$ they are not isomorphic as Lie algebras. The examples as described in~\eqref{it:th2} were previously discovered in \cite{Nik}.

We can be more precise about the allowed inner products. On $\ag$ it is given, and on $V$ we can choose any {\it admissible} inner product, i.e. one such that $ (\ad_N)_{|\cW}$ is skew-symmetric for all $N \in \so(\ag)$ which normalize $V$ (or more generally, for a subalgebra of such elements that still satisfies the above Definition of a GO-pair). One such inner product on $\cW\subset \so(\ag)$ is given by a negative multiple of the restriction of the Killing form of $\so(\ag)$. We call such inner products \emph{standard}. For the GO-pairs in Theorem~\ref{th:class} we have:
\begin{itemize}
	\item If $V$ is of centralizer type, then any inner product on $V$ is admissible.
	\item If $\dim V=3$ in case~\eqref{it:th3rep}, the inner product is standard.
	\item If $\dim V =6$ or $7$, then the restriction of the inner product to $\Span(J_1, J_2, J_3, J_4, J_5)$ is standard and $J' \perp J_a$, respectively  $J_6,J' \perp J_a$ for $a=1, \dots, 5$.
\end{itemize}
For GO $2$-step nilpotent Lie algebras of a Clifford type, the admissible inner products were classified in~\cite[Theorem~5.10]{Lau}.

 We point out that  for the GO pairs of centralizer type, our proof gives rise to a classification without the assumption of being n0n-singular, see  Corollary~\ref{c:GOct}. For $\dim\z=1$ this is well known, and for $\dim\z=2$ one finds another proof in~\cite[Theorem~6.1]{Lau}.
\smallskip

The structure of the proof is as follows. From earlier analysis, we will be able to assume that the inner product on $V$ is standard. We then show that the classification reduces to the case of irreducible GO-pairs, i.e., ones where `the action' of  $V$ on $\ag$ is irreducible, and show how the general GO-pairs  can be obtained by `gluing' the irreducibles. The main part of the proof is thus the classification of irreducible GO-pairs. After obtaining certain bounds on $\dim\ag$ and $\dim V=\dim\z$, we get a short list of possible configurations. A case by case study then finishes the proof.
\smallskip

The paper is organized as follows. In Section~\ref{s:prelim} we collect necessary facts on GO-pairs, some of which hold without the assumption of being non-singular.  Then in Section~\ref{s:examples} we consider the three most interesting examples in Theorem~\ref{th:class}, namely \eqref{it:th3rep}, \eqref{it:th6} and \eqref{it:th7}. We prove that they are indeed GO and determine the admissible inner products. In Section~\ref{s:proofdim8} we give the proof of Theorem~\ref{th:class} in the most important case, when $\dim \ag = 8$ assuming certain dimension bounds which we postpone to later. Section~\ref{s:irst} contains the remainder of the proof of Theorem~\ref{th:class}.

\section{Preliminaries. Structure of GO-pairs}
\label{s:prelim}

\subsection{Normalizer, centralizer and reducibility}
\label{ss:nc}

Let $(\ag, \ip)$ be a Euclidean vector space with inner product $\ip$  and $\so(\ag)$ the Lie algebra of skew-symmetric matrices. For a linear subspace $\cV \subset \so(\ag)$, we define the centralizer $\cC(\cV)$ and the normalizer $\cN(\cV)$ of $\cV$ by

\begin{itemize}
	\item $\cC(\cV)=\{A\in\so(\ag)\mid [A,V]=0\}$
		\item $\cN(\cV)=\{A\in\so(\ag)\mid [A,V]\subset V\}$
		\item $ \cN(\cV)=\cC(\cV)\oplus \cP(\cV)$,
\end{itemize}
where the latter is the direct sum of ideals orthogonal with respect to the Killing form of $\so(\ag)$. Notice that by the Jacobi identity, both $\cC(\cV)$ and $\cN(\cV)$ are subalgebras of $\so(\ag)$ and that $\cC(\cV)$ is  an ideal in $\cN(\cV)$, and hence its orthogonal complement is an ideal as well. Moreover, both $\ag$ and the subspace $\cV \subset \so(\ag)$ are $\cP(\cV)$- and $\cN(\cV)$-modules, and the representation of $\cP(\cV)$ on $\cV$ is faithful.

Now assume that $\cV$ is equipped with an inner product, which for simplicity of notation, we again denote by $\ipr$. We then define the \emph{skew-symmetric normalizer}:
 $$\Ng(\cV)=\{A\in\so(\ag)\mid [A,V]\subset V \text{ and } (\ad_N)_{|\cV}:\cV \to \cV \text{  is skew-symmetric} \}$$
 	with respect to the inner product $\ipr$.
 Notice that $\cC(V)$ is an ideal in $\Ng(\cV)$. We again define the orthogonal sum of ideals:
 $$
 \Ng(\cV)=\cC(\cV)\oplus \Pg(\cV)
 $$
  We call $\Pg(\cV)$ the \emph{pure normalizer} of $\cV$, relative to $\ipr$. Note that $\Pg(\cV)$  is a subalgebra in $\cP(\cV)$, and is acting faithfully on $\cV$. The inner product on $V$ that defines  $\Ng(\cV)$ and  $\Pg(\cV)$ will be clear from context.

If the inner product $\ipr$ on $\cV$ is negatively proportional to the restriction of the Killing form of $\so(\ag)$ to $\cV$ we call it \emph{standard}. For a standard inner product we clearly have $\Ng(\cV)=\cN(\cV)$ and $\Pg(\cV)=\cP(\cV)$.

We finally make the following observation. The connected Lie groups $\sP$ and $\sC$ with Lie algebras $\cN(\cV)$ and $\cC(\cV)$ respectively are compact since they are stabilizer and centralizer groups and hence act properly on both $\ag$ and $V$. The same is true for the Lie group with Lie algebra $\Pg(\cV)$, although it may not be compact,  since any Lie group that acts by isometries in some metric is a proper group action. Thus we can apply, e.g. the principle isotropy Lemma to these actions.

We note the following two simple facts.

\begin{lemma} \label{l:sum}
  Let $\cV$ be a subspace of $\so(\ag)$ with the inner product $\ipr$. In the above notation we have
  \begin{enumerate} [label=\emph{(\alph*)},ref=\alph*]
    \item \label{it:cNcCm}
    $\Ng(\cV) \subset \cN (\cV) \subset \cC(\cV) + \m(\cV)$, where $\m(\cV)$ is the subalgebra of $\so(\ag)$ generated by $\cV$.
  \item \label{it:CgCg}
    $[\cC(\cV), \Pg(\cV)] = [\cC(\cV),\cP(\cV)]=0$.
  \end{enumerate}
\end{lemma}
\begin{proof}
For assertion~\eqref{it:cNcCm}, the inclusion $\Ng(\cV) \subset \cN(\cV)$ follows from the definition. Furthermore, $\cV$ and $\m(\cV)$ have the same centralizer $\cC(\cV)$, and $\Ng(\cV)$ lies in the normalizer $\cN(\m(\cV))$ of $\m(\cV))$. But $\cN(\m(\cV)) = \m(\cV) + \cC(\cV)$. Indeed, the inclusion $\cN(\m(\cV)) \supset \m(\cV) + \cC(\cV)$ is obvious. On the other hand, if $N \subset \cN(\m(\cV)) \setminus \m(\cV)$, then $\m(\cV) \oplus \br N$ is a subalgebra of $\so(\ag)$ containing a subalgebra of codimension one. As both subalgebras are reductive, this is only possible when $\m(\cV) \oplus \br N = \m(\cV) \oplus \br N'$ for some $N'$ which centralises $\m(\cV)$.

Assertion \eqref{it:CgCg} follows from the fact that $\cC(\cV)$ and $\Pg(\cV)$, respectively $\cC(\cV)$ and $\cP(\cV)$, are complementary ideals in $\Ng(\cV)$, respectively in $\cN(\cV)$.
\end{proof}

Although the subspace $\cV \in \so(\ag)$ is rarely a subalgebra, the space $\ag$ enjoys the property of complete $\cV$-reducibility and even a version of  Schur's Lemma.

For a subspace $L \subset \ag$, we denote by $\so(L)$ the subalgebra of $\so(\ag)$ consisting of the elements of $\so(\ag)$ which annihilate $L^\perp \subset \ag$, and by $\pi_L: \so(\ag) \to \so(L)$ the orthogonal projection relative to the Killing form of $\so(\ag)$.
We now make the important definition:
\begin{equation*}
L\subset\ag \text{ is called \emph{$\cV$-invariant}, if } JL \subset L \text{ for all } J \in \cV.
\end{equation*}
We say that $\ag$ is \emph{$\cV$-irreducible}, if any $\cV$-invariant proper subspace of $\ag$ is trivial. Clearly if $L \subset \ag$ is $\cV$-invariant, then so is $L^\perp$. It follows that the space $\ag$ can be decomposed into the orthogonal sum of $\cV$-irreducible subspaces: $\ag =\oplus_{i=1}^p \ag_i$.

We abbreviate the notation $\pi_{\ag_i}$ to $\pi_i$, and denote $\cV_i=\pi_i(\cV) \subset \so(\ag_i)$. When $\cV$ is equipped with an inner product $\ipr$, we define the inner product $\ipr_i$ on $\cV_i$ by push-forward of $\ipr$ under $\pi_i$.

For $i=1, \dots, p$, we denote $\Ng(\cV_i) \subset \so(\ag_i)$ the skew-symmetric normalizer of $\cV_i \subset \so(\ag_i)$ relative to the inner product $\ipr_i$, and $\cC(\cV_i)$ its centralizer. Then $\cC(\cV_i)$ is again an ideal in $\Ng(\cV_i)$, and we define the pure normalizer $\Pg(\cV_i) \subset \so(\ag)$ as the orthogonal complement to $\cC(\cV_i)$  in $\Ng(\cV_i)$ relative to the Killing form of $\so(\ag_i)$.

For $i \ne j, \; i,j=1, \dots, p$, we introduce the subspace $\h_{ij} \subset \so(\ag)$ as the orthogonal complement to $\so(\ag_i) \oplus \so(\ag_j)\subset\so(\ag_i \oplus \ag_j)$, relative to the Killing form,  and denote by $\pi_{ij}: \so(\ag) \to \h_{ij}$ the orthogonal projection to $\h_{ij}$. Finally, let $\cC_{ij}=\pi_{ij}(\cC) \subset \h_{ij}$.

We then have the following.

\begin{lemma} \label{l:red}
  Let $\cV \subset \so(\ag)$ be a linear subspace and $\ipr$ be an inner product on $\cV$. Let $\ag =\oplus_{i=1}^p \ag_i$ be an orthogonal decomposition of $\ag$ into $\cV$-irreducible subspaces. In the above notation, the following holds.

  \begin{enumerate}[label=\emph{(\alph*)},ref=\alph*]
    \item \label{it:redCdec}
    There is a direct, orthogonal decomposition $\cC(\cV) = \oplus_{i=1}^p \cC(\cV_i) \oplus \oplus_{1 \le i < j \le p} \cC_{ij}$.

    \item \label{it:redCdiva}
   The subspace $\cC(\cV_i) \oplus \br \, \id_{\ag_i} \subset \gl(\ag_i)$ is an associative division algebra,
    in particular, $\dim \cC(\cV_i) \in \{0,1,3\}$.
    If  $\dim\cC(\cV_i)=1$, it is the linear span of a complex structure on $\ag_i$, hence  $\cV_i, \Pg(\cV_i)\subset \mathfrak{u}(\ag_i)$.
     If  $\dim\cC(\cV_i)=3$, it is isomorphic to $\spg(1)$ and gives a quaternionic structure on $\ag_i$. Thus  $\cV_i, \Pg(\cV_i)  \subset\spg(\ag_i)$.

    \item \label{it:redCij}
    Suppose that $\cC_{ij} \ne 0$ for some $i<j$. Then $\dim \ag_i = \dim \ag_j$, and there exist orthonormal bases for $\ag_i$ and $\ag_j$ relative to which the matrices of $\pi_i(J)$ and $\pi_j(J)$ are equal, for all $J \in \cV$. Moreover, relative to these bases in  $\ag_i \oplus \ag_j$, the subspace $\cC_{ij} \subset \h_{ij}$ is given by $\cC_{ij} = \br \left(\begin{smallmatrix}0&I\\-I&0\end{smallmatrix}\right) \oplus \Span(\left(\begin{smallmatrix} 0&C\\C&0 \end{smallmatrix}\right) \, | \, C \in \cC(\cV_i))$.

    \item \label{it:redP}
     $\Pg(\cV) \subset \oplus_{i=1}^p \Pg(\cV_i) \subset \oplus_{i=1}^p \so(\ag_i) \subset \so(\ag)$, and is isomorphic to a subalgebra of $\so(\cV)$.
  \end{enumerate}
\end{lemma}
\begin{proof}
  The required decomposition in assertion~\eqref{it:redCdec} follows, as the subspaces $\ag_i$ are $\cV$-invariant, hence $\cV \subset \oplus_{i=1}^p \so(\ag_i)$.

  For assertion~\eqref{it:redCdiva}, take $C \in \cC(\cV_i)$. Then $C^2$ commutes with $\cV_i$, and hence must be a multiple of the identity, as $\ag_i$ is $\cV$-irreducible. As $\cC(\cV_i)$ is also a Lie subalgebra, the subspace $\cC(\cV_i) \oplus \br \, \id_{\ag_i} \subset \gl(\ag_i)$ is closed under matrix multiplication and under taking the inverse. Thus $\cC(\cV_i) \oplus \br \, \id_{\ag_i} \subset \gl(\ag_i)$ is a division algebra and hence isomorphic to one of $\br, \bc$ or $\bh$, by Frobenius Theorem. The rest of the claim is straightforward.

  To prove assertion~\eqref{it:redCij} it is sufficient to prove it when $p=2$ and $i=1, \, j=2$. Suppose that a nonzero matrix $A=\left(\begin{smallmatrix} 0&T\\-T^t&0 \end{smallmatrix}\right)$ (with the block decomposition corresponding to the decomposition $\ag=\ag_1 \oplus \ag_2$) commutes with $J= \left(\begin{smallmatrix} \pi_1(J)&0\\0&\pi_2(J) \end{smallmatrix}\right)$, for all $J \in \cV$. Then the same is true for the matrix $-A^2=\left(\begin{smallmatrix} TT^t&0\\0&T^tT \end{smallmatrix}\right)$. As both $\ag_1$ and $\ag_2$ are $\cV$-irreducible, both $TT^t$ and $T^tT$ must be multiples of the identity matrices. It follows that $\dim \ag_1 = \dim \ag_2$, and we denote their common value by $m$. Thus $T$  is, up to a nonzero multiple, an orthogonal $m \times m$ matrix. We can therefore choose orthonormal bases for $\ag_1$ and $\ag_2$ relative to which $A= \left(\begin{smallmatrix}0&I_{m}\\-I_{m}&0\end{smallmatrix}\right)$ (up to a nonzero multiple), and hence the matrices of $\pi_1(J)$ and $\pi_2(J)$ are equal, for all $J \in \cV$. Now if $B=\left(\begin{smallmatrix} 0&Q\\-Q^t&0 \end{smallmatrix}\right)$, where $Q$ is an $m \times m$ matrix with $\Tr Q =0$, lies in $\cC_{12}$, then $Q$ commutes with $\pi_1(J)$, for all $J \in \cV$, and so do the symmetric and the skew-symmetric parts of $Q$. As $\ag_1$ is $\cV$-irreducible (and $\Tr Q = 0$), the matrix $Q$ must be skew-symmetric, and hence $Q \in \cC(\cV_1))$.

  For assertion~\eqref{it:redP}, take $N \in \Ng(\cV)$. As $\cV \subset \oplus_{i=1}^p \so(\ag_i)$, the projection of $N$ to the orthogonal complement of $\oplus_{i=1}^p \so(\ag_i)$ in $\so(\ag)$ lies in $\cC(\cV)$. Since $\Pg(\cV)$ is the orthogonal complement to $\cC(\cV)$ in $\Ng(\cV)$, it lies in $\oplus_{i=1}^p \so(\ag_i)$. Moreover, as the subalgebra of $\cC(\cV)$ lying in $\oplus_{i=1}^p \so(\ag_i)$ is given by $\oplus_{i=1}^p \cC(\cV_i)$ by assertion~\eqref{it:redCdec}, we obtain that $\pi_i(\Pg(\cV)) \subset \Pg(\cV_i)$.
\end{proof}

\subsection{The GO condition}
\label{ss:GOc}
We start with the following three definitions, which will be used throughout the paper.

\begin{definition*} %\label{def:anyip}
Let $(\cV, \ipr)$ be an inner product space with $\cV \subset \so(\ag)$.
 \begin{enumerate} [label=(\alph*),ref=\alph*]
	\item \label{it:defa}
    $(\cV, \ipr)$  is called a \emph{GO-pair} if for every $J \in \cV$ and every $X \in \ag$, there exists an $N = N(J,X) \in \Ng(\cV)$ such that
		\begin{gather} \label{eq:NJ}
			[N,J]=0 \quad \text{and}\\ JX=NX. \label{eq:NXJX}
		\end{gather}
	\item  $(\cV, \ipr)$ is called a \emph{GO subspace}, if~\eqref{it:defa} holds with $\Ng(\cV)$ replaced by $\cN(\cV)$.
	\item
    $(\cV, \ipr)$  is called a GO-pair of \emph{centralizer type} if~\eqref{it:defa} holds with $\Ng(\cV)$ replaced by $\cC(\cV)$. We then call $\cV$ a \emph{subspace of centralizer type}; recall that $\cC(\cV)$ does not depend on the inner product $\ipr$ on $\cV$.
	\end{enumerate}
\end{definition*}

In the terminology of~\cite[Definition~2.11]{Gor}, what we call a GO subspace is called a subspace satisfying the \emph{transitive normalizer condition}. The motivation for the  definition is explained in Theorem~\ref{tha:gorcond} in  the Introduction: there is a one-to-one correspondence between the GO-pairs and the geodesic orbit nilmanifolds.  Recall also that the data $(\cV, \ipr)$ with $V\subset \so(\ag)$ is   in one-to-one correspondence with metric 2-step nilpotent Lie algebras. Two subspaces $V_1, V_2 \subset\so(\ag)$ give rise to isomorphic Lie algebras iff the they are conjugate, and to isometric metric $2$-step nilpotent Lie algebras iff the conjugacy induces an isometry of the corresponding inner products on $V_i$.

Notice that $(\cV, \ipr)$ being a GO subspace is equivalent to saying that $(\cV, \ipr)$ is a GO-pair with respect to the restriction of the Killing form of $\so(\ag)$. In particular, if $(\cV,\ipr)$ is GO-pair for some choice of an inner product on $\cV$, then $\cV$ is a GO subspace.  The converse may clearly be false: if $\Ng(\cV)$ is a proper subalgebra of $\cN(\cV)$, it may not have ``sufficiently many" elements to satisfy the GO conditions. It follows that to classify GO-pairs one may start with classifying GO subspaces $\cV \subset \so(\ag)$ and then finding inner products $\ipr$ on them such that the GO property is preserved with respect to the subalgebra  $\Ng(\cV) \subset \cN(\cV)$. Such inner products are called \emph{admissible}. Finally, notice that if $(\cV,\ipr)$ is of centralizer type, then any product on $V$ is admissible since  $\cC(\cV)$ does not depend on the inner product. Examples and classifications of classes of GO-pairs with nonstandard inner product are given in~\cite[Theorem~5.10]{Lau} and in \cite[Theorems~1, 2]{Nik}.

\begin{remark} \label{rem:ipretc}
We note the following facts.
\begin{enumerate}[label=(\alph*),ref=\alph*]
  \item \label{it:iprcommonker} % ~\ref{def:anyip}
  If all $J \in \cV$ have a common kernel $\cs$, then the corresponding $2$-step nilpotent  Lie algebra $\n$ decomposes into the orthogonal sum of a smaller nilpotent algebra and an abelian ideal. In this case,  we can replace $\ag$ by $\ag'= \ag \cap \cs^\perp$ in the Definition without violating the property of $(\cV, \ipr)$  being a GO-pair \cite[Lemma~3]{Nik}.

  \item \label{it:iprsubalg}
  A subalgebra of $\cV\subset \so(\ag)$  is clearly a GO subspace since we can  simply choose  $N(J,X)=J$. Note that  $\cV$ being a subalgebra  is equivalent to  $\cV \subset \cN(\cV)$. If in addition the inner product on $\cV$ is admissible, which means that it is a bi-invariant metric, then $(\cV, \ipr)$ being a GO-pair is equivalent to the corresponding GO nilmanifold being \emph{naturally reductive}, see \cite[Theorem~2.5]{Gor}).
\end{enumerate}
\end{remark}

In the ``generic" case this is all one can get, as the following lemma shows.
\begin{lemma}\label{l:generic}
  Suppose $\cV \subset \so(\ag)$ contains a regular element of $\so(\ag)$. If $(\cV,\ipr)$ is a GO-pair, then $\cV \subset \Ng(\cV)$ and hence a subalgebra of $\so(\ag)$ .
\end{lemma}
\begin{proof}
  Let $J \in \cV$ be regular. Then any $N\in\so(\ag)$ with $[N,J]=0$  belongs to the Cartan subalgebra $\t \subset \so(\ag)$ defined by $J$. Choosing $X \in \ag$ in such a way that its component in every 2-dimensional $J$-invariant subspace is non-zero, it is easy to see that  $NX=JX$ can  only be satisfied if $N=J$ and hence $J \in \Ng(\cV)$. As the set or regular elements in $\so(\ag)$ is open, its intersection with $\cV$, which is nonempty, is relatively open in $\cV$. Repeating the argument, it follows that $\Ng(\cV)$ contains an open subset of $\cV$, and hence the whole subspace $\cV$. Thus also $V\subset \cN(\cV) $, and hence $V$ is a subalgebra.
\end{proof}

Passing to a $\cV$-invariant subspace preserves the GO property, as the following lemma shows.

\begin{lemma} \label{l:redst}
  Let $(\cV, \ipr)$ be a GO-pair, where $\cV \subset \so(\ag)$.
  \begin{enumerate}[label=\emph{(\alph*)},ref=\alph*]
    \item \label{it:redV}
    If $L \subset \ag$ is a $\cV$-invariant subspace, then $(\pi_L(\cV), \ipr_L)$ is a GO-pair.  Moreover, if $\cV \subset \so(\ag)$ is a GO subspace of centralizer type, then so is $\pi_L(\cV) \subset \so(L)$. If $V$ is non-singular, then $\pi_L(\cV)$ is non-singular as well, and $\pi_L\colon V\to \pi_L(\cV) $ is an isomorphism.

    \item \label{it:redN}
    Any $\Ng(\cV)$-invariant subspace of $\ag$ is also $\cV$-invariant, and so if $\ag$ is $\cV$-irreducible, then it is also $\Ng(\cV)$-irreducible.

    \item \label{it:redPir}
    Suppose $\ag$ is $\cV$-irreducible. Although it is $\Ng(\cV)$-irreducible by~\eqref{it:redN}, it can be $\Pg(\cV)$-reducible under its action on $\ag$. Let
    $\ag=\oplus_{i=1}^p \ag_i$, be a decomposition into irreducible $\Pg(\cV)$-modules. Then the modules $\ag_i$ are isomorphic, and we have the following cases:
    \begin{itemize}
    	\item $p=1$ if $\ag_i$ is of quaternionic type, or of complex type with $\cC(\cV) \subset \so(2)$, or of real type with $\cC(\cV)=0$.
    	\item $p=2$ if $\ag_i$ are of  complex type with $\cC(\cV) \cong \spg(1)$, or of real type with $\cC(\cV)=0$.
    	\item $p=4$  if $\ag_i$ are of  real type with $\cC(\cV) \cong \spg(1)$.
    \end{itemize}
  \end{enumerate}
\end{lemma}
\begin{proof}
  For~\eqref{it:redV} we note that $\pi_L(\Ng(\cV))$ normalizes $\pi_L(\cV)$ and the action of the elements of $\pi_L(\Ng(\cV))$ on  $\pi_L(\cV)$ is skew-symmetric relative to $\ipr_L$. Moreover,
  replacing $J$  with $\pi_L(J)$ and $N$  with $\pi_L(N)$, the GO condition is preserved.
   The second claim follows by replacing $\Ng(\cV)$ with $\cC(\cV)$ in the previous argument. For the last claim, observe that if $J\in V$ and $X\in L$ with $JX=0$, then $J\in\so(\ag)$ is singular.

  For assertion~\eqref{it:redN}, let $L$ be an $\Ng(\cV)$-invariant subspace. Then for any   $X \in L$ and $J\in V$ there exists $N\in \Ng(\cV) $ such that $JX=NX \in L$ and hence $L$ is $V$-invariant.

  For part~\eqref{it:redPir}, assume that   the representation of $\Pg(\cV)$ contains non-isomorphic irreducible submodules. Then there exists a decomposition $\ag=\ag' \oplus \ag''$ into $\Pg(\cV)$-submodules such that no irreducible submodule of $\ag'$ is isomorphic to an irreducible submodule of $\ag''$. Since $\cC(\cV)$ commutes with  $\Pg(\cV)$, both subspaces $\ag'$ and $\ag''$ are $\cC(\cV)$-invariant. Hence they are also $\Ng(\cV)$-invariant, which contradicts part~\eqref{it:redN}.

 If  $\ag_i$ are of quaternionic type, then the centralizer of $\Pg(\cV)$ in $\so(\ag)$ is isomorphic to $\spg(p)$ and contains the subalgebra   $\cC(\cV)$. Since $\ag$ is $V$-irreducible, Lemma \ref{l:red} implies that  $\cC(\cV)$  lies in the sum of $p$ copies of the standard representation of $\spg(1)$ on $\br^4$. Thus the representation of $\Ng(\cV)=\Pg(\cV) \oplus \cC(\cV)$ is reducible, which contradicts assertion~\eqref{it:redN} unless $p=1$. Similarly, suppose $\ag_i$ is of complex type. The centralizer of $\Pg(\cV)$ in $\so(\ag)$ is isomorphic to $\ug(p)$ and contains a subalgebra $\cC(\cV)$. If $p=1$, then $\dim\cC(\cV) \le 1$. If $p > 1$, then $\cC(\cV)$ must be nontrivial. But if $\cC(\cV) \cong \so(2)$, or if $\cC(\cV) \cong \spg(1)$ and $p > 2$, then $\ag$ is $\Ng(\cV)$-reducible. Suppose $\ag_i$ is of real type. The centralizer of $\Pg(\cV)$ is isomorphic to $\so(p)$ and contains $\cC(\cV)$, and so in order for $\ag$ to be $\Ng(\cV)$-irreducible, we need to have either $p=1$ and then $\cC(\cV)=0$, or $p=2$ and  $\cC(\cV) \cong \so(2)$, or $p=4$ and then $\cC(\cV) \cong \spg(1)$.
\end{proof}

\subsection{GO subspaces of centralizer type}
\label{ss:centr}

In this section, we classify subspaces $\cV \subset \so(\ag)$ of centralizer type (without the assumption of being non-singular). Recall that this means that for any $J \in \cV$ and any $X \in \ag$ there exists $N \in \cC(\cV)$ such that $JX=NX$. Thus in this case any inner product is admissible. Clearly, if $\dim V=1$ or $2$, then any  $N\in \Ng(\cV)$ satisfying the GO condition with a non-zero $J\in V$ lies in the centralizer $ \cC(\cV)$ since $(\ad_N)_{|V}$ is skew-symmetric and contains $J$ in its kernel.  Hence any GO-pair with $\dim V=1, 2$ is of centralizer type, see~\cite[Section~6.2]{Lau}, \cite[Proposition~6]{Nik}.

Consider the following construction.  Choose an orthogonal decomposition of $\ag$ as $\ag=\bigoplus_{i=1}^pU_i\oplus\bigoplus_{j=1}^q W_j\oplus Z $ with $\dim U_i=4$ and $\dim W_j=2$.
Thus $\so(U_i)\simeq\so(4)\simeq \s_i\oplus\s_i'$ where $\s_i$ and $\s_i'$ are the two $\so(3)$ ideals of $\so(4)$. Define the subalgebra $\f \subset \so(\ag)$ by $\f=\bigoplus_{i=1}^p \s_i \oplus \bigoplus_{j=1}^q \so(W_j)$. We call any so constructed subalgebra $\f \subset \so(\ag)$ \emph{nice}.

\begin{proposition} \label{p:ct}
  A  GO subspace $\cV \subset \so(\ag)$ is of centralizer type if and only if it is a subspace of a nice subalgebra.
\end{proposition}
\begin{proof}
For the ``if" part, observe that if $\cV\subset \f$, then  $\cC(\cV)$ contains  the subalgebra $\f'=\bigoplus_{i=1}^p \s_i' \oplus \bigoplus_{j=1}^q \so(W_j)$.  To show that $\cV$ is of centralizer type it is sufficient to check that for any $A \in \s_i$ and any $X \in U_i$ there exists a $Q \in \s_i'$ such that $AX=QX$. But $AX \perp X$, and for $X \ne 0$, the map $Q \mapsto QX$ from $\s_i'$ to $X^\perp$ is surjective, as $Q^2$ is a multiple of the identity for any nonzero $Q \in \s_i'$.

To prove the ``only if" part, we first claim  that $\cV$ is a GO subspace of centralizer type if and only if $\m(\cV)$ is, where  $\m(\cV) \subset \so(\ag)$ is the subalgebra generated by $\cV$. In one direction, this follows from the fact that $\cV$ and $\m(\cV)$ have the same centralizer. In the other direction, suppose that $\cV$ is GO of centralizer type. Let $J_1, J_2 \in \cV, \; X \in \ag$, and let $N_1, N_2 \in \cC(\cV)$ be such that $J_1X=N_1X$ and $J_2X=N_2X$. Then $[J_1,J_2]X = J_1N_2X-J_2N_1X = N_2J_1X-N_1J_2X = N_2N_1X-N_1N_2X = [N_2, N_1]X$, and so $\cV + [\cV,\cV]$ is a again a GO subspace of centralizer type. Repeating the argument we find that $\m(\cV)$ is also a subspace of centralizer type.

It is therefore sufficient to prove the ``only if" part under the assumption that $\cV \subset \so(\ag)$ is a subalgebra. We construct the nice algebra $\f$ containing $\cV$ as follows. Let $L \subset \ag$ be an invariant, $\cV$-irreducible subspace. If $\dim L =1$, it is a trivial $\cV$-module; the sum of those gives us the subspace $Z$ in the definition of $\f$. The irreducible subspaces $L$ of dimension $2$ are the subspaces $W_j$ in the definition of $\f$. Suppose $\dim L > 2$. By Lemma~\ref{l:redst}\eqref{it:redV} the projection $\pi_L(\cV) \subset \so(L)$ is still GO of centralizer type (and is a subalgebra). Its centralizer $\cC_L \subset \so(L)$ can be of one of three types given in Lemma~\ref{l:red}\eqref{it:redCdiva}. For any choice of $J' \in \pi_L(\cV)$ and any $X \in L$ there exists a $C \in \cC_L$ such that $CX=J'X$. It follows that $\cC_L$ cannot be trivial, and moreover, $(J')^2X=C^2X=\mu X$ for some $\mu \le 0$. Hence $(J')^2$ is a multiple of the identity, for all $J' \in \pi_L(\cV)$, hence $L$ is a Clifford module. Now if $[J_1',J_2']=0$ for some nonzero, orthogonal $J_1',J_2' \in \pi_L(\cV)$, then $0=J_1'J_2'- J_2'J_1'=J_1'J_2'+ J_1'J_2'$, and so $J_1'J_2'=0$ contradicting the fact that $(J')^2$ is a nonzero multiple of the identity for all nonzero $J' \in \pi_L(\cV)$. It follows that the (reductive) algebra $\pi_L(\cV)$ has rank one. It cannot be $1$-dimensional, as $\dim L > 2$, and hence $\pi_L(\cV)$ is isomorphic to $\spg(1)$. Then $L$ is a $\mathrm{Cliff}(2)$-module and so $\dim L = 4$ since $L$ is also irreducible. Thus $\pi_L(\cV)$ is the image of the defining representation of $\spg(1)$, and hence one of the ideals in $\so(4)$. Such a subspace $L$ is a subspace $U_i$ in the definition of the nice algebra $\f$.
\end{proof}
As a simple consequence, we can now describe the classification of non-singular GO-pairs (or GO subspaces) of centralizer type. For this, let  $\ag=\bigoplus_{i=1}^pU_i$ with $\dim U_i=4$ and  $\so(U_i)\simeq\so(4)\simeq \s_i\oplus\s_i'$  as above.
\begin{corollary} \label{c:GOct}
	If $(\cV, \ipr)$ is a non-singular GO-pair of centralizer type, then it belongs to one of the following:
	\begin{enumerate}[label=\emph{(\alph*)},ref=\alph*]
		\item \label{it:dim1}
		If $\dim V=1$, then $V$ is spanned by a non-singular skew-symmetric matrix.
		
		\item \label{it:dim23}
        If $\dim V \in \{2, 3\}$, then $V \subset \bigoplus_{i=1}^p \s_i$ and every projection $\pi_i:V \to \s_i$ is injective.
	\end{enumerate}
\end{corollary}
This, together with the fact that the inner product on $V$ is arbitrary, gives a large class of GO nilmanifolds of centralizer type. The examples with $\dim V=2$ were discovered by  Y.Nikonorov in \cite{Nik}.
\begin{proof}
	The claim for $\dim V=1$ is clear. If $L$ is one of the $V$-invariant subspaces in the definition of $\f$, then $\pi_L(V)$ and $V$ have the same dimension. Hence if $1 < \dim V\le 3$ then $W_j=Z=\{0\}\}$.
\end{proof}
\smallskip

The following lemma gives some  sufficient conditions for a GO subspace to be of centralizer type, which will be useful for us later on.

\begin{lemma} \label{l:forc}
    Let $(\cV, \ipr)$ be a GO-pair, where $\cV \subset \so(\ag)$. In each of the following cases:
    \begin{enumerate}[label=\emph{(\alph*)},ref=\alph*]
        \item \label{it:large}
        The principal isotropy subalgebra of the action of $\Pg(\cV)$ on $\cV$ is trivial,

        \item \label{it:cpab}
        The pure normalizer $\Pg(\cV)$ is abelian,
  \end{enumerate}
  the subspace $\cV$ is of centralizer type.
\end{lemma}
\begin{proof} Since the Lie group $\sP$ with Lie algebra $\Pg(\cV)$ acts properly on $V$,  the set of regular points $V_{reg}\subset V$ is open and dense. Thus if $J\in V_{reg}$, the Lie algebra of the principle isotropy group, which is given by $\{N\in \Pg(\cV)\mid [N,J]=0 \}$, is trivial. Thus it lies in the ineffective kernel $\cC(\cV)$ and hence  $[N,J]=0$ implies that  $N\in \cC(\cV)$. Thus  for any $X \in \ag$ we have $JX \in \{NX \,| \, N \in \cC(\cV)\}$. As this is satisfied for a generic $J \in \cV$, it is satisfied for an arbitrary $J \in \cV$, and so $\cV$ is of centralizer type.

Assertion~\eqref{it:cpab} follows from the fact that the principle isotropy group of an effective action of an abelian group is trivial.
\end{proof}
Note that the cases where the action of $\sP$ is not almost free, i.e. the principle isotropy group is not finite, are relatively rare. The classification of such actions of a semisimple compact Lie group is given in \cite{HH}. For example, if $\sP$ is  simple, the list contains only the defining representations of the classical Lie groups $\SO(n), \SU(n), \Sp(n)$, the adjoint representations, some $s$-representations, and a small number of exceptional ones.
Thus, if the GO-pair $(\cV, \ipr)$ is not of centralizer type, the representation of $\Pg(\cV)$ on $V$ is already quite special.

Another useful application of GO subspaces of centralizer type is as follows.

\begin{lemma} \label{l:W}
Let $(\cV,\ipr)$ be a GO-pair and let $J \in V_{reg}$ be a regular element of the action of $\sP$ on $V$. Denote $\cC(J)$ the centralizer of $J$ in $\Ng(\cV)$ and $W \subset \cV$ the subspace of elements which are centralized by $\cC(J)$. Then:

\begin{enumerate}[label=\emph{(\alph*)},ref=\alph*]
    \item \label{it:Wct}
    The subspace $W \subset \so(\ag)$ is a GO subspace of centralizer type.

    \item \label{it:Worbit} % Lie  \subset \SO(\cV)
    The orbit of $W$ under the action of the group $\sP$ on $\cV$ is the whole $\cV$, and $\sP$ acts on $\cV$ with cohomogeneity $\dim W$.
\end{enumerate}
Under additional assumption that $V$ is non-singular we have the following:
\begin{enumerate}[label=\emph{(\alph*)},ref=\alph*,resume]
  \item \label{it:W123}
  The group $\sP$ acts on $\cV$ with cohomogeneity $\dim W \in \{1,2,3\}$.

  \item \label{it:Wsimple}
  If $J$ has a simple eigenvalue, then  $W= \br J$ and $\sP$ acts transitively (and almost effectively) on the unit sphere of $\cV$.
\end{enumerate}
\end{lemma}
\begin{proof}
  \eqref{it:Wct} Notice that all the elements $J' \in W$ which are close to $J$ have the same centralizer $\cC(J')=\cC(J)$ in $\Ng(\cV)$. Choose a basis $\{J_s\}$ for $W$ made of such elements, and fix $X \in \ag$. By~\eqref{eq:NXJX}, for any $s=1, \dots, \dim W$, we can find an element $N_s \in \cC(J)$ which satisfies $N_sX=J_sX$, and so for an arbitrary $J' \in W$, an element $N \in \cC(J)$ which satisfies~\eqref{eq:NXJX} can be constructed by linearity. Hence $W$ is a GO subspace of centralizer type.

  \eqref{it:Worbit} follows from the Principal Orbit theorem and assertion~\eqref{it:Wct}.

  \eqref{it:W123} follows from Corollary~\eqref{c:GOct}, as $W$ is non-singular.

  \eqref{it:Wsimple} From Corollary~\eqref{c:GOct}\eqref{it:dim23}, no element of $W$ has a simple eigenvalue when $\dim W =2, 3$. So $\dim W =1$, and hence by~\eqref{it:W123} the group $\sP$ acts on $V$ with cohomogeneity $1$.
\end{proof}

\subsection{GO-pairs of Clifford type}
\label{ss:cliff}

We say that a subspace $\cV \subset \so(\ag)$ is of \emph{Clifford type} (also called \emph{H-type}), if $J^2$ it a multiple of the identity $\id_{\ag}$ for any $J \in \cV$. Clearly, any such subspace $\cV$ is non-singular and is spanned by $\dim \cV$ pairwise anticommuting complex structures which can be constructed from representations of the corresponding Clifford algebra on $\ag$.

We say that a GO-pair $(\cV,\ipr)$ is of \emph{Clifford type}, if the subspace $\cV$ of Clifford type. GO-pairs of Clifford type are classified in \cite{Rie} for the standard inner product, and in \cite[Theorem~5.10]{Lau} for an arbitrary inner product $\ipr$:

\begin{theorema} [{\cite{Rie},\cite[Theorem~5.10]{Lau}}] \label{tha:Cliff}
  Let a subspace $\cV \subset \so(\ag)$ be of Clifford type. A GO-pair $(\cV, \ipr)$ is of Clifford type precisely in the following cases.
  \begin{enumerate}[label=\emph{(\alph*)},ref=\alph*]
    \item \label{it:Cliff123}
    $\dim \cV \in \{1,2,3\}$ and $\ipr$ is arbitrary.

    \item \label{it:Cliff567}
    $\dim \cV \in \{5,6,7\}, \; \dim \ag = 8$, and $\ipr$ is such that its restriction to a $5$-dimensional subspace of $V$ is standard.

    \item \label{it:Cliff73}
    $\dim \cV = 7, \; \dim \ag \in \{16, 24\}$, the inner product $\ipr$ is standard, and $\cV$ is spanned by $7$ pairwise anticommuting complex structures $K_i$ such that $K_1 \dots K_7 = \pm \id_\ag$.
  \end{enumerate}
\end{theorema}

Note that the GO-pairs $(\cV, \ipr)$ in Theorem~\ref{tha:Cliff}\eqref{it:Cliff123} are of centralizer type (see Section~\ref{ss:centr}).

In the following table, we collect information on the dimensions of spin representations of the algebras $\so(n)$ (which will be the most common case in what follows), their type, and the maximal dimensions $d$ of non-singular subspaces of skew-symmetric matrices in the corresponding representation spaces given by~\eqref{eq:RH}.

{\renewcommand{\arraystretch}{1.4}
    \begin{table}[h!]
    \caption{Real spin representations and dimensions of non-singular subspaces.}
    \centering
    \begin{tabular}{|c|c|c|c|c|c|c|c|c|c|c|c|c|c|c|}
      \hline
        & $\Delta_3$ & $\Delta_4^\pm$ & $\Delta_5$ & $\Delta_6$ & $\Delta_7$ & $\Delta_8^\pm$ & $\Delta_9$ & $\Delta_{10}$ & $\Delta_{11}$ & $\Delta_{12}^\pm$ & $\Delta_{13}$ & $\Delta_{14}$ & $\Delta_{15}$ & $\Delta_{16}^\pm$\\
      \hline
      space & $\br^4$ & $\br^4$ & $\br^8$ & $\br^8$ & $\br^8$ & $\br^8$ & $\br^{16}$ & $\br^{32}$ & $\br^{64}$ & $\br^{64}$ & $\br^{128}$ & $\br^{128}$ & $\br^{128}$ & $\br^{128}$\\
      \hline
      type  & $\bh$ & $\bh$ & $\bh$ & $\bc$ & $\br$ & $\br$ & $\br$ & $\bc$ & $\bh$ & $\bh$ & $\bh$ & $\bc$ & $\br$ & $\br$\\
      \hline
      $d$  & $3$ & $3$ & $7$ & $7$ & $7$ & $7$ & $8$ & $9$ & $11$ & $11$ & $15$ & $15$ & $15$ & $15$\\
      \hline
    \end{tabular}
    \label{tab:spin}
    \end{table}
}

\subsection{Non-singular GO-pairs of Rep type}
\label{ss:gorep}

Recall that a subspace $\cV \subset \so(\ag)$ is called of \emph{Rep type}, if it is a subalgebra of $\so(\ag)$. If $\cV$ is also non-singular and $\dim \cV > 1$, then $\dim \cV = 3$ and $\cV$ is the image of the representation of $\so(3)$ on $\ag$ all of whose irreducible sub-representations $\ag_i\subset\ag$ are of dimensions divisible by $4$ (see e.g. \cite[Section~5.1]{LO}). The GO condition is always satisfied for a standard inner product on $V$, with $N=J$. Moreover, for any admissible inner product, the pure normalizer $\Pg(\cV)$ is a subalgebra of $\so(V)\simeq\so(3)$ by Lemma~\ref{l:red}\eqref{it:redP}. If $\Pg(\cV)=\so(3)$, then any invariant inner product is standard. Otherwise $\Pg(\cV)$ is abelian and hence Lemma~\ref{l:forc}\eqref{it:cpab} implies that $\cV$ is of centralizer type. Then by Corollary~\ref{c:GOct} $\dim \ag_i=4$ for all $i$ and any inner product is admissible. In the case when $\dim \ag_i>4$ for some $i$, the only admissible inner product is standard.

\section{Examples}
\label{s:examples}

In this Section we discuss the remaining examples in Theorem~\ref{th:class}. We already covered the cases of the examples of centralizer type. It remains to discuss case~\eqref{it:th3rep} of $\dim V=3$, and cases $\dim V=6,7$ with $\dim \ag = 8$.

\subsection{\ensuremath{\texorpdfstring{\dim}{dim} V = 3}}
\label{ss:dim3}

An equivalent description of the GO-pairs in case~\eqref{it:th3}\eqref{it:th3rep} of Theorem~\ref{th:class} is as follows. For $i=1,\ldots, p$, let $\rho_i : \so(3) \to \so(\ag_i)$ be irreducible representations with $\dim \ag_i = 4m_i$, and $\la_i$ nonzero constants such that $\la_i = 1$ if $m_i > 1$. The subspace $\cV \subset\oplus_{i=1}^p \so(\ag_i)$ is spanned by the three elements $J_a$ with $\pi_i(J_a)=\la_i \rho_i(E_a)$, where $E_1,E_2,E_3$ is a basis of $\so(3)$. We will now show that it is indeed a GO-pair, with respect to the standard inner product. We can assume that $V$ is not of centralizer type, and hence $m_i > 1$ for at least one $i$.

Notice that  $V_i=\rho_i(\so(3))$ is a subalgebra of $\so(\ag_i)$. Hence, since $\cP(\cV_i)$ acts effectively on $V_i$,  it follows that $\cP(\cV_i)=V_i  =\rho_i(\so(3))  \subset \so(\ag_i)$. By Lemma~\ref{l:forc}\eqref{it:cpab}, the pure normalizer $\cP(\cV)$ cannot be abelian. Thus $\cP(\cV) \cong \so(3)$ since it is isomorphic to a subalgebra of $\so(\cV) = \so(3)$.  It now follows from Lemma~\ref{l:red}\eqref{it:redP} that $ \cP(\cV) \subset \oplus_{i=1}^p \cP(\cV_i)=\oplus_{i=1}^p \rho_i(\so(3))=\rho(\so(3))$ where $\rho=\oplus_{i=1}^p \rho_i$ is the direct sum of the representations $\rho_i$. Thus $ \cP(\cV) =\rho(\so(3))$.

If $m_i=1$, we have $\so(\ag_i)\simeq \so(4)= \s_i\oplus\s_i'$ as the direct sum of ideals. The irreducible representation $\rho_i$ is the inclusion of $\spg(1)$ into $\so(4)$. Thus we can assume that $\rho_i(V_i)=\s_i$ and hence $\cC(\cV_i) =\s_i'$.

Take an arbitrary $J=\sum_{i=1}^p \la_i\rho_i(E) \in \cV$, with  $E \in \so(3)$, and an arbitrary $X = \sum_{i=1}^{p} X_i \in \ag$, with $X_i \in \ag_i$. Let  $N = \rho(E) + \sum_{i: m_i=1} C_i\in  \cN(V)$, where $C_i \in \cC(\cV_i)$ if $m_i=1$. Any such $N$ normalizes $\cV$ and commutes with $J$. Then the second GO  condition $JX=NX$ is equivalent to $\la_i\rho_i(E)X_i = \rho_i(E)X_i$ when $m_i > 1$ and to $(\la_i-1)\rho_i(E)X_i=C_i X_i$ when $m_i=1$. The former is clearly satisfied, and for the latter we can assume $\lambda_i \ne 1$ since otherwise we  choose $C_i=0$. Next we observe that for any $X_i \in \ag_i$, which we can assume to be non-zero, we have  $C_i X_i \in X_i^\perp$ and the map from $\cC(\cV_i)=\s_i'$ to $X_i^\perp \subset \ag_i$ defined by $C \mapsto CX_i$ is surjective. Thus we can always choose $C_i$ in order to satisfy the second equation.

Finally, we claim that any admissible inner product is standard. Indeed, if $\ip$ is admissible, then $\Pg(V)\subset \cP(V)\simeq \so(3)$. We can again assume that $\Pg(V)$ is not $1$-dimensional since $V$ is not of centralizer type, and hence $\Pg(V)= \cP(V)$. The inner product $\ip$ being admissible means that for any $N\in \cP(V)$, the endomorphism $(\ad N)_{|V}$ is skew-symmetric in $\ip$. We can write $\<J,J'\>= (SJ,J')$   where $(\ , \ )$ is a standard inner product and $S$  a positive definite symmetric linear map on $\cV$. Since $(\ad N)_{|V}$ is skew-symmetric in $(\, \ )$,  it is skew-symmetric in $\ip$ if and only if it commutes with $S$. Since $\cP(V)$ acts irreducibly on $V$, it follows that $S$ is a multiple of the identity and hence $\ip$ is standard as well.

\subsection{\ensuremath{\texorpdfstring{\dim}{dim} \texorpdfstring{\ag}{\unichar{"1D51E}} = 8}}
\label{ss:dima8}

Consider the standard representation $\pi$ of the algebra $\spg(2)$ on $\br^8$. Then $\Lambda^2\pi$ splits into the following sum of $\spg(2)$-modules: the adjoint module, which can be viewed as the subalgebra $\spg(2) \subset \so(8)$, the three-dimensional trivial module, which is the subalgebra $\spg(1) \subset \so(8)$ which centralizes $\spg(2)$, and the sum of three isomorphic $5$-dimensional modules $\br^5$ on which $\spg(2)$ acts via  the standard representations of $\so(5) \cong \spg(2)$.

We can describe these modules explicitly as follows. We let $\bh$ be the algebra of quaternions, $\Im \bh$ the imaginary quaternions, and denote by $\overline{q}$ the conjugate of $q$. For $q \in \bh$, let $L_q, R_q: \bh \to \bh$ be  the left and  right multiplication by $q$, respectively.

If we identify $\br^8$ with $\bh^2$, then
relative to an orthonormal basis for $\br^8 = \bh^2$, the above $\spg(2)$-submodules of $\so(8)$ are given as follows:
\begin{equation}\label{eq:sp2modules1}
	\begin{gathered}
		\spg(2)=\bigg\{\begin{pmatrix}
			L_a & L_q \\
			L_{-\overline{q}} & L_b
		\end{pmatrix} \, | \, a, b \in \Im \bh,\, q \in \bh\bigg\}, \quad
		\spg(1)=\bigg\{\mathcal{R}_c = \begin{pmatrix}
			R_c & 0_4 \\
			0_4 & R_c
		\end{pmatrix} \, | \, c \in \Im \bh\bigg\}, \\
		\br^5_d=\bigg\{\begin{pmatrix}
			R_d & 0_4 \\
			0_4 & R_d
		\end{pmatrix}
		\begin{pmatrix}
			\mu I_4 & L_p \\
			L_{\overline{p}} & -\mu I_4
		\end{pmatrix} \, | \, \mu \in \br,\, p \in \bh\bigg\}\ \text{\ with }\  0\ne d\in\Im \bh ,
	\end{gathered}
\end{equation}

For a subspace $W \subset \Im \bh$, we denote $\spg(1)_W=\Span(\mathcal{R}_c \, | \, c \in W)$ in the notation of~\eqref{eq:sp2modules1} and  introduce the $\spg(2)$-module
\begin{equation}\label{VdW}
	\cV_{d,W} = \br^5_d \oplus \spg(1)_W\ \text{with}\ d\in\Im\bh, \ |d|=1\quad \text{and } W \subset \Im \bh
\end{equation}

\begin{remark} \label{rem:J67}
We observe that for $\theta \in (0, \pi/2)$, the subspaces $\cV_{d,W} \subset \so(8)$ with $W= \br\, \mathrm{j}$ and $d=\cos \theta \mathrm{i} + \sin \theta \mathrm{j}$ are precisely those in Theorem~\ref{th:class}\eqref{it:th6}, and  $\cV_{d,W} $ with $W= \Span(\mathrm{j},\mathrm{k})$ and $d=\cos \theta \mathrm{i} + \sin \theta \mathrm{j}$, are those in Theorem~\ref{th:class}\eqref{it:th7}. To see this we notice that any element of $\br^5_d$ as given in~\eqref{eq:sp2modules1} squares to $-\mu^2 \|d\|^2 \|p\|^2 I_8$, and so $\br^5_d$ is spanned by $5$ anticommuting complex structures $J_1, \ldots, J_5$ on $\br^8$. Moreover, $J_6=\mathcal{R}_{\mathrm{k}}$ and $J_7=\mathcal{R}_{-\sin \theta \mathrm{i} + \cos \theta \mathrm{j}}$ are anticommuting complex structures which anticommute with $J_1, \ldots, J_5$ as well, and $J'=J_7 \cos \theta + J_6J_7 \sin \theta = \mathcal{R}_{\mathrm{j}}$.
\end{remark}

In order to show that these subspaces have the properties as claimed, we prove the following general statement.

\begin{lemma}\label{l:sp2R8}
	Let  $\cV_{d,W} = \br^5_d + \spg(1)_W \subset \so(8)$  as defined in \eqref{VdW}. Then
	\begin{enumerate}[label=\emph{(\alph*)},ref=\alph*]
		\item \label{it:sp2R8GO}
		$\cV_{d,W} \subset \so(8)$ is an irreducible GO subspace.
		
			\item \label{it:sp2R8nonsing}
		$\cV_{d,W}$ is non-singular iff $d \notin W$.
		
		\item \label{it:sp2R8types}
		$\cV_{d,W}$  is of  Clifford type iff  $d \perp W$,
		
		\item \label{it:sp2R8normalizer}
		If $d \notin W$ and $d \not\perp W$, then $\cC(\cV_{d,W})=0$ and $\cN(\cV_{d,W})=\cP(\cV_{d,W})=\spg(2)$. Moreover, for almost all pairs $(J,X)$ an element $N=N(J,X) \in \cN(\cV_{d,W})$ satisfying the GO condition is unique and the set of such elements $N(J,X)$ is open in $\cN(\cV_{d,W})$.
		
		\item \label{it:sp2R8rep}
			$\cV_{d,W}$ is never of Rep type or centralizer type.
	
	\end{enumerate}
\end{lemma}

\begin{proof}
\eqref{it:sp2R8GO}	First  note that the subalgebra $\spg(2)$ normalizes  $\cV_{d,W}$, and so it is sufficient to prove that for any $J \in \br^5_d + \spg(1)$ and $X \in \br^8$, the GO conditions  are satisfied for some $ N \in \spg(2)$. Since the action of $\spg(2)$ on $\br^5_d$ is the standard action of $\so(5)$ on $\br^5$, and $\Sp(2)$ acts transitively on the unit sphere, we can conjugate by an element of $\Sp(2)$ and hence assume  that $p=0$. Thus	it suffices to prove the GO conditions with the choice
	\begin{equation} \label{eq:sp2R8J1}
		J = \begin{pmatrix}
			R_d & 0_4 \\
			0_4 & R_d
		\end{pmatrix}
		\begin{pmatrix}
			\mu I_4 & 0_4 \\
			0_4 & -\mu I_4
		\end{pmatrix} +
		\begin{pmatrix}
			R_c & 0_4 \\
			0_4 & R_c
		\end{pmatrix} =
		\begin{pmatrix}
			R_{\mu d+c} & 0_4 \\
			0_4 & R_{-\mu d + c}
		\end{pmatrix} \quad c\in W, \ \mu\in \br.
	\end{equation}
	An element of $\spg(2)$ as given by~\eqref{eq:sp2modules1} commutes with  $J$ when $\mu q = 0$, and so taking $N \in \spg(2)$ with $q=0$ we have $[J,N]=0$. Then for a vector $X=(x,y)^t,\; x, y \in \bh$, the condition $JX=NX$ is equivalent to the existence of $a, b \in \Im \bh$ such that the two quaternionic equations $x(\mu d + c) = a x$ and $y(-\mu d + c) = b y$ are satisfied. For the first equation, if $x=0$, there is nothing to prove, and if $x \ne 0$, we take $a= x(\mu d + c) x^{-1}$. Notice  that $a \in \Im \bh$ follows, as $c,d \in \Im \bh$ and $\mu \in \br$. A similar argument applies to the second equation. This shows that  $\cV_{d,W}$ is a GO subspace.
	
\eqref{it:sp2R8types} Any $J\in \cV_{d,W}$  is conjugate to a matrix $J$ given by~\eqref{eq:sp2R8J1}, which can be singular only if $\pm \mu d + c = 0$ for some $\mu \in \br$ and some $c \in V$. So $\cV_{d,W}$ is non-singular if and only if $d \notin W$.
	
\eqref{it:sp2R8nonsing} Note that the eigenvalues of $J^2$ are $-(\mu^2 + \|c\|^2 \pm 2 \mu \<c,d\>)$, which are  equal only if $\<c,d\>=0$ when $\mu \ne 0$. Thus $J$ is a multiple of the identity, i.e.,   $\cV_{d,W}$ is of Clifford type, only if $d \perp W$.
	
\eqref{it:sp2R8normalizer} Recall that  $\spg(2) \subset \cP(\cV_{d,W}) \subset \so(8)$. On the other hand, the only intermediate subalgebras are $ \spg(2) \subset  \spg(2) \oplus \br  \subset  \spg(2) \oplus \spg(1) \subset \so(8)$, or $\spg(2)\subset \su(4)  \subset \so(8)$. The latter is excluded since it does not centralize $\spg(1)$. In the former case one easily sees that no element of $\spg(1)$ normalizes $\cV_{d,W}$, using the fact that $d \notin V$ and $d \not \perp V$. Thus $\cN(\cV_{d,W})=\cP(\cV_{d,W})=\spg(2)$ and $\cC(\cV_{d,W})=0$. Furthermore,  the calculation in the proof of~\eqref{it:sp2R8GO} shows that for $J$ given by~\eqref{eq:sp2R8J1} and $X=(x,y)^t,\; x, y \in \bh \setminus \{0\}$, there is a unique $N \in \cN(\cV)$ which satisfies the GO condition, namely $N = \left(\begin{smallmatrix} L_a & 0\\ 0 & L_b \end{smallmatrix}\right)$, where $a= x(\mu d + c) x^{-1}$ and $b = y(-\mu d + c) y^{-1}$. Conjugating by $\Sp(2)$ we obtain that the set of such elements $N$ is an open subset of $\cN(\cV) = \spg(2)$.
	
\eqref{it:sp2R8rep}  $\cV_{d,W}$ is not a subalgebra of $\so(8)$ since	$[\cV_{d,W},\cV_{d,W}] \supset [\br^5_d, \br^5_d]=\spg(2)$, hence no $\cV_{d,W}$ is of Rep type. It is not of centralizer type either since those satisfy $\dim V\le 3$.
\end{proof}
	
Notice that in both cases $\dim V =6$ and $\dim V = 7$ in Theorem~\ref{th:class}, the subspaces $\cV_{d,W}$ with different $\theta \in (0, \pi/2)$ are not conjugate, even by $\SO(8)$. Indeed, from the proof of~\eqref{it:sp2R8nonsing} and by Remark~\ref{rem:J67}, the maximum of the maximal eigenvalue of the matrix $J^2$, where $J \in \cV$ with $\Tr (J^2)=-1$, equals $\frac18(\sin \theta -1)$. This implies that the nilmanifolds constructed from them are pairwise non-isomorphic.

We next determine admissible inner products. From Lemma~\ref{l:sp2R8}\eqref{it:sp2R8normalizer} we know that $\Ng(\cV) \subset \cN(\cV) = \spg(2)$ and that the set of elements $N \in \cN(\cV)$ needed to satisfy the GO conditions is an open subset of $\cN(\cV) = \spg(2)$. Thus an inner product $\ipr$ is admissible if and only if $\Ng(\cV) = \cN(\cV) = \spg(2)$. This means that the action of $\spg(2)$ on $\cV_{d,W}$ has to be skew-symmetric relative to $\ipr$. This action is the standard action of $\so(5) \cong \spg(2)$ on $\br^5_d$ and is trivial on $\spg(1)_W$ in the notation of~\eqref{VdW}. Hence an inner product is admissible if and only if its restriction to $\br^5_d$ is standard and $\br^5_d \perp \spg(1)_W$. This is equivalent to the condition in the Introduction by the construction in Remark~\ref{rem:J67}.

\section{Proof of Theorem~\ref{th:class} for \ensuremath{\texorpdfstring{\dim}{dim} \texorpdfstring{\ag}{\unichar{"1D51E}} \texorpdfstring{\le}{\unichar{"2264}} 8}}
\label{s:proofdim8}

The most interesting and important case of Theorem~\ref{th:class} is the case $\dim \ag =8$. We give the proof of this case here assuming the dimension bounds from Table~\ref{tab:dims1} which we prove later in Proposition~\ref{p:dimbound}. We discuss the case $\dim \ag < 8$ in the end of the section.

\begin{lemma} \label{l:dima8}
Let $\cV \subset \so(\ag)$ be a non-singular GO subspace such that $\dim \ag =8$ and the space $\ag$ is $\cV$-irreducible. Then we have one of the following:
	\begin{enumerate}[label=\emph{(\alph*)},ref=\alph*]
		\item \label{it:dim8Cl}
		$(\cV, \ipr)$ is of Clifford type, as given in Theorem~\ref{tha:Cliff}\eqref{it:Cliff567}.
		
		\item \label{it:dim8Rep}
		$(\cV, \ipr)$ is of Rep type, and $\cV$ is the image of an irreducible representation of $\so(3)$ on $\br^8$, see Section~\ref{ss:gorep}.
		
		\item \label{it:dim8J'}
		$(\cV, \ipr)$ is as given in Theorem~\ref{th:class}\eqref{it:th6} or \eqref{it:th7}.
\end{enumerate}
\end{lemma}
\begin{proof}
Assume that $\cV \subset \so(8)$ is a non-singular GO subspace which is of neither Rep, nor Clifford type. Then $V$ contains an open and dense set of elements $J$ such that $J \notin \cN(\cV)$ and that $J^2$ has at least two different eigenvalues. Choose $J$ in the intersection of these two sets and also belonging to the set $V_{reg}$ of regular elements of the action of $\sP$ on $V$.

The centralizer $\cC(J)$ of $J$ in $\cN(\cV)$ is a subalgebra of the centralizer $\cC_{\so(8)}(J)$ of $J$ in $\so(8)$. Moreover, we have $\cC(J)=\cC_P(J) \oplus \cC(\cV)$, the direct sum of ideals, where $\cC_P(J) = \cC(J) \cap \Pg(\cV)$, and where $\cC(\cV)$ is given in Lemma~\ref{l:red}\eqref{it:redCdiva}.

\medskip

From Table~\ref{tab:dims1}, case $18$, we have $\dim \cC(J) \ge 6$ and $4 \le \dim V \le 7$.

\medskip

We first consider the case when $J$ has a simple eigenvalue. Choosing an orthonormal basis for $\br^8$ we obtain
	\begin{equation}\label{eq:R8simple}
		J=\left(
		\begin{array}{ccc}
			0 & \a & 0 \\
			-\a & 0 & 0 \\
			0 & 0 & K \\
		\end{array}
		\right), \quad
		\cC(J) \subset \s = \Big\{
		\left(
		\begin{array}{ccc}
			0 & t & 0 \\
			-t & 0 & 0 \\
			0 & 0 & U \\
		\end{array}
		\right)
		\, | \, t \in \br, \, U \in \ug(3)\Big\} \cong \br^2 \oplus \su(3),
	\end{equation}
where $\a \ne 0$ and $K$ is a non-singular $6 \times 6$ matrix having no eigenvalues $\pm \a \, \mathrm{i}$. We note that the subalgebra $\s$ does not contain a subalgebra $\spg(1)$ consisting of matrices which square to a multiple of the identity, and so the centralizer $\cC(\cV)$ in the decomposition $\cC(J)=\cC_P(J) \oplus \cC(\cV)$ can be either trivial, or $1$-dimensional. As $\dim \cC(J) \ge 6$, we obtain that $\dim \cC_P(J) \ge 5$. We also note that $\dim \cC_P(J) \le \dim \cC(J) < 10$, as $J \notin \cC(J)$, and so $\cC(J)$ is a proper subalgebra of $\s$.
	
	Furthermore, by Lemma~\ref{l:W}\eqref{it:Wsimple}, the group $\sP \subset \SO(\cV)$ acts transitively and almost effectively on the unit sphere of $\cV$. The principal isotropy subalgebra of this action is $\cC_P(J)$. From \cite[Table~I]{Str}, when $4 \le \dim \cV \le 7$ and $5 \le \dim \cC_P(J) \le 9$, this is only possible in the following two cases: $\SO(5)$ acting on $\br^5$ with $\cC_P(J)=\so(4)$ and $G_2$ acting on $S^6$ with $\cC_P(J)=\su(3)$. But the first case is not possible as the algebra $\s$ given in~\eqref{eq:R8simple} does not contain a subalgebra $\so(4)$, and in the second case we arrive at a contradiction with Lemma~\ref{l:redst}\eqref{it:redPir}, as $\g_2$ has no nontrivial representation whose dimension divides $8$.
	
\medskip

Now suppose $J$ has no simple eigenvalues. Then choosing an orthonormal basis for $\br^8$ we obtain
	\begin{equation*}
		J=\left(
		\begin{array}{cc}
			\a K & 0_4 \\
			0_4 & \b K \\
		\end{array}
		\right), \quad
		\cC(J) \subset \s = \Big\{
		\left(
		\begin{array}{cc}
			U_1 & 0_4 \\
			0_4 & U_2 \\
		\end{array}
		\right)
		\, | \, U_1, U_2 \in \ug(2)\Big\} \cong \br^2 \oplus \su(2) \oplus \su(2),
	\end{equation*}
	where $\a \ne \b,\, \a, \b \ne 0$, and $K \in \so(4)$ is such that $K^2=-I_4$, and $U_1, U_2 \in \so(4)$ commute with $K$. Similar to the previous case we note that the centralizer $\cC(\cV)$ in the decomposition $\cC(J)=\cC_P(J) \oplus \cC(\cV)$ cannot be isomorphic to $\spg(1)$. Indeed, the only subalgebra of $\s$ isomorphic to $\spg(1)$ and consisting of matrices whose square is a multiple of $I_8$ is the diagonal of $\su(2) \oplus \su(2)$, up to conjugation, and then its centralizer in $\s$ is the abelian ideal $\br^2 \subset \s$. Then both $J$ and $\cC_P(J)$ lie in this ideal, which contradicts the fact that $\dim \cC(J) \ge 6$. Thus $\cC(\cV)$ is either trivial, or $1$-dimensional, and as $\dim \cC(J) \ge 6$, we obtain that $\dim \cC_P(J) \ge 5$. On the other hand, $\dim \cC_P(J) \le \dim \cC(J) < 8$, as $J \notin \cC(J)$, and so $\cC(J)$ is a proper subalgebra of $\s$.
	
	By Lemma~\ref{l:W}\eqref{it:W123}, the subgroup $\sP \subset \SO(\cV)$ acts on $\cV$ (almost effectively and) with cohomogeneity $c \in \{1,2,3\}$. From~\cite[Tables I, II and III]{Str} we obtain that the only possible cases for such action when $4 \le \dim \cV \le 7$ and $5 \le \dim \cC_P(J) \le 7$ are as follows: either $\cP(\cV) = \so(5)$ acting by its standard representation on $\br^5$, $\dim \cV = 5, 6, 7$ (in the latter two cases, $\cP(\cV)$ acts trivially on $(\br^5)^\perp \subset \cV$), or $\cP(\cV) = \so(5) \oplus \so(2)$ acting by its standard representations on $\cV = \br^7= \br^5 \oplus \br^2$. In both cases, $\cC_P(J) = \so(4)$. Then by Lemma~\ref{l:redst}\eqref{it:redPir}, the representation of the ideal $\so(5) \subset \cP(\cV)$ on $\br^8$ is the standard representation of $\spg(2)$. But then the subspace $\cV$ (of dimension at least $5$) is an $\spg(2)$-submodule in $\so(8)$, and so in the notation of~\eqref{eq:sp2modules1} and Lemma~\ref{l:sp2R8} we obtain $\cV = \cV_{d,V}$, where $d$ is a unit imaginary quaternion and $V \subset \Im \bh$. Any such subspace $\cV_{d,V} \subset \so(8)$ is GO by Lemma~\ref{l:sp2R8}\eqref{it:sp2R8GO}, and then if we want to avoid $\cV_{d,V}$ to be singular (Lemma~\ref{l:sp2R8}\eqref{it:sp2R8nonsing}) or to be of Clifford type Lemma~\ref{l:sp2R8}\eqref{it:sp2R8types}, we obtain one of the spaces as in Theorem~\ref{th:class}(\ref{it:th6}, \ref{it:th7}) --- see Remark~\ref{rem:J67} in Section~\ref{s:examples}.
\end{proof}

In case $\dim \ag < 8$, we obtain that either $\dim \cV \le 2$, or $\dim \ag = 4$ and $\dim \cV = 3$ -- see~\eqref{eq:RH} below. In both cases the subspace $\cV$ is of centralizer type: when $\dim \cV \le 2$, we know this from Section~\ref{ss:centr}, and when $\dim \ag = 4$ and $\dim \cV = 3$, all such subspaces are known from \cite[Example~2.14]{Gor}. Hence we obtain the GO subspaces from cases~\eqref{it:th1}, \eqref{it:th2} or \eqref{it:th3}\eqref{it:th3cen} of Theorem~\ref{th:class}.

\section{Proof of Theorem~\ref{th:class}}\label{s:irst}

\subsection{Preparation}
\label{ss:prep}
First of all we note that the `if' direction of Theorem~\ref{th:class} is already proved in the earlier Sections. Indeed, the Clifford type is fully covered by Theorem~\ref{tha:Cliff}. Furthermore, the pairs in cases \eqref{it:th1}, \eqref{it:th2} and \eqref{it:th3}\eqref{it:th3cen} of Theorem~\ref{th:class} are of centralizer type. The fact that they are indeed GO is established in Section~\ref{ss:centr}. The pairs in Theorem~\ref{th:class}~\eqref{it:th3}\eqref{it:th3rep} are shown to be GO in Section~\ref{ss:dim3}, and for the pairs Theorem~\ref{th:class}~\eqref{it:th6}\eqref{it:th7} in  Lemma~\ref{l:dima8}.

Second, for the rest of the proof we can (and will) assume that \emph{the inner product $\ip$ on $V$ is standard}, and so $V$ is a GO subspace in our terminology. Indeed, suppose Theorem~\ref{th:class} is already established in the assumption that $\ip$ is standard and we want to find all admissible inner products. In cases \eqref{it:th1}, \eqref{it:th2} and \eqref{it:th3}\eqref{it:th3cen} of Theorem~\ref{th:class}, the subspace $V$ is of centralizer type, and so any inner product is admissible. The fact that only standard inner product is admissible for subspaces in Theorem~\ref{th:class}~\eqref{it:th3}\eqref{it:th3rep} is shown in Section~\ref{ss:dim3}. Admissible inner products for subspaces in Theorem~\ref{th:class}~\eqref{it:th6} and~\eqref{it:th7} are found in the end of Section~\ref{ss:dima8}.

Third, the non-singular GO subspaces of centralizer, Clifford and Rep type are fully understood --- see Sections~\ref{ss:centr}, \ref{ss:cliff} and \ref{ss:gorep}, respectively. For brevity, we call a GO subspace of neither of these three types a \emph{GO subspace of common type}. If $V$ is of common type, we call an element $J \in \cV$ \emph{generic}, if
  \begin{enumerate}[label=(\roman*),ref=\roman*]
    \item \label{it:genJnotinN}
    $J \notin \cN(\cV)$,

    \item \label{it:genJmaxeigen}
    $J^2$ has the maximal number of pairwise distinct eigenvalues among the elements of $\cV$ (at least two, by our assumption), and

    \item \label{it:genJreg}
    $J$ is a regular element of the action of the identity component $\sN \subset \SO(\ag)$ of the normalizer of $\cV$ in $\SO(\ag)$ (note that $\sN$ is compact and hence this action is proper).
  \end{enumerate}
Notice that generic elements form an open and dense subset of $\cV$. Moreover, their centralizers of generic elements in $\cN(\cV)$ are conjugate by elements of $\sN$ by the Principal Orbit Theorem.

Furthermore, we recall that the maximal dimension of a non-singular subspace $\cV \subset \so(\ag)$ (not necessarily GO) is controlled by the Radon-Hurwitz number: if $\dim \ag = 2^{4b+c}u$, where $u$ is odd and $0 \le c \le 3$, then
\begin{equation}\label{eq:RH}
  \dim \cV \le 2^c + 8 b - 1.
\end{equation}

We also make the following simple, but useful observation.

\begin{lemma} \label{l:dimkerN}
  Suppose $(\cV,\ipr)$ is a GO-pair, where the subspace $\cV \subset \so(\ag)$ is non-singular. If $\ag$ is $\cV$-irreducible, then for any nonzero element $N \in \Ng(\cV)$, we must have $\dim \Ker N \le \frac12 \dim \ag$.
\end{lemma}
\begin{proof}
  Suppose that the kernel of a nonzero $N \in \Ng(\cV)$ is of dimension greater than $\frac12 \dim \ag$. Such $N$ cannot lie in $\cC(\cV)$ (as by Lemma~\ref{l:red}\eqref{it:redCdiva} all nonzero elements of $\cC(\cV)$ square to a multiple of the identity), and so there exists $J \in \cV$ such that the element $[N, J]$ (which still lies in $\cV$) is nonzero. But then $[N, J]$ must be singular, a contradiction.
\end{proof}

\medskip

The proof of Theorem~\ref{th:class} goes as follows. Suppose we are given a Euclidean space $(\ag, \ip)$ and a non-singular subspace $\cV \subset \so(\ag)$ with a standard inner product $\ipr$. If $L \subset \ag$ is $\cV$-invariant, then by Lemma~\ref{l:red}\eqref{it:redV}, the pair $(\pi_L (\cV), \ipr_L)$ is also GO, and it is clear that the subspace $\pi_L (\cV) \subset \so(L)$ is non-singular as well. This reduces the classification of non-singular GO spaces to the case when $\ag$ is $\cV$-irreducible. As soon as this classification is known, `gluing' irreducible subspaces back together can be done using Lemma~\ref{l:red} --- see Section~\ref{ss:redu}.

The classification of GO subspaces $\cV$, where $\ag$ is $\cV$-irreducible is given in Proposition~\ref{p:irst} below. The core of the proof is Proposition~\ref{p:dimbound} in Section~\ref{ss:dim}, in which from the study of the action of the Lie group $\sN \subset \SO(\ag)$ which normalizes $\cV$, we establish an upper bound for the dimension of $\ag$ is in terms of the dimension of $\cV$. Combining this upper bound with the lower bound coming from~\eqref{eq:RH} we obtain a finite list of possible cases for the dimensions of $\ag$ and of $\cV$, as well as the possible candidates for the pure normalizer $\cP(\cV)$ and the centralizer $\cC(\cV)$ in each case (recall that by Lemma~\ref{l:red}\eqref{it:redCdiva}, for $\cV$-irreducible $\ag$, the centralizer $\cC(\cV)$ is either trivial or is spanned by a complex or a quaternionic structure on $(\ag,\ip)$). This list is given in Table~\ref{tab:dims1}. In Section~\ref{ss:cases} we go through this list and show that if $\dim \ag > 8$, then $V$ may only be of either Clifford, or centralizer or Rep type. The proof of Proposition~\ref{p:irst} is completed by Lemma~\ref{l:dima8} which treats the case $\dim \ag = 8$.

\medskip

In the next two sections we prove the following.

\begin{proposition} \label{p:irst}
  Suppose $\cV \subset \so(\ag)$ is a non-singular GO subspace of common type such that the space $\ag$ is $\cV$-irreducible (and the inner product on $\cV$ is standard). Then $\dim\ag=8$ and $\cV$ is one of the spaces in Theorem~\ref{th:class} \eqref{it:th6} or \eqref{it:th7}.
\end{proposition}

Note that in the assumptions of Proposition~\ref{p:irst}, we can assume that $\dim \ag \ge 8$. Furthermore, the proof of Proposition~\ref{p:irst} in case $\dim \ag = 8$ is given in Lemma~\ref{l:dima8}, \emph{assuming} certain inequalities which we prove in the next section.

Recall that in the assumptions of Proposition~\ref{p:irst}, $\ag$ is an irreducible $\cN(\cV)$-module, by Lemma~\ref{l:redst}\eqref{it:redN}. From Lemma~\ref{l:red}\eqref{it:redCdiva} we know that $\cC(\cV)$, if it is nontrivial, is spanned by a complex or by a quaternionic structure on $(\ag,\ip)$, and then both $\cV$ and $\cP(\cV)$ lie in the corresponding subalgebras $\su(\frac12 \dim \ag)$ and $\spg(\frac14 \dim \ag)$ of $\so(\ag)$, respectively.

\subsection{Dimension bounds}
\label{ss:dim}

Denote $\sN, \sC$ and $\sP$ the connected Lie subgroups of $\SO(\ag)$ with the Lie algebras $\cN(\cV), \cC(\cV)$ and $\cP(\cV)$, respectively. Note that $\sN$ is the (almost) direct product of $\sC$ and $\sP$, and that all three subgroups are compact (as the normalizer and the stabilizer subgroups of any subset of $\ag$ are compact).

For $J \in \cV$, denote $\cC(J)$ its centralizer in $\cN(\cV)$. Note that $\cC(J)=\cC_P(J) \oplus \cC(\cV)$ (orthogonal decomposition into ideals), where $\cC_P(J)$ is the centralizer of $J$ in the pure normalizer of $\cV$, that is, $\cC_P(J) = \cC(J) \cap \cP(\cV)$. At the group level, we consider the subgroup $\sC(J)$ which is the identity component of the intersection of the stabilizer of $J$ in $\SO(\ag)$ with $\sN$. Note that $\sC(J)$ is also compact.

We now assume that $J$ is generic and denote $\overline{\cC}(J) = \cC(J) \oplus \br J$. Let $\overline{\sC}(J) \subset \SO(\ag)$ be the connected subgroup with the Lie algebra $\overline{\cC}(J)$. We make the following observation.

\begin{observation*} %\label{obs:proper}
  At the Lie algebra level, the algebra $\overline{\cC}(J)$ is the direct sum of the ideal $\cC(J)$ and the $1$-dimensional ideal $\br J$. At the Lie group level, $\overline{\sC}(J)$ is the (almost) direct product of the compact group $\sC(J)$ and the $1$-dimensional subgroup $\{\exp (tJ) \, | \, t \in \br\} \subset \SO(\ag)$. If all the eigenvalues of $J$ are pairwise rationally dependent, then the latter subgroup is compact, as also is $\overline{\sC}(J)$, and hence $\overline{\sC}(J)$ acts properly on $\ag$. However if they are not, then $\overline{\sC}(J) \cong \sC(J) \times \br$, and then the action of $\overline{\sC}(J)$ on $\ag$ is improper. Luckily, the GO property comes to the rescue here. Indeed, by~\eqref{eq:NXJX}, for any $X \in \ag$, we can find $N \in \cC(J)$ such that $NX=JX$, and as $N$ and $J$ commute, it follows that $\exp (tN)X = \exp (tJ) X$, for all $t \in \br$. Thus, although $J \notin \cC(J)$, the orbit $\exp (tJ) X$ lies on the orbit $\sC(J) X$, for all $X \in \ag$, and so the (potentially improper) action of $\overline{\sC}(J)$ and the proper action of the compact group $\sC(J)$ on $\ag$ are \emph{orbit-equivalent}: they have the same orbits. Following the standard proofs, we see that the action of $\overline{\sC}(J)$ enjoys the same nice properties as a proper action, of which we will need the following two.

  Firstly, the Principal Orbit Theorem: there is a principal orbit type and the union of principal orbits is open and dense in $\ag$; we call the points belonging to that union regular. Secondly, the normal slice representation of the stationary subalgebra of $\overline{\cC}(J)$ at a regular point is trivial. To see this, we notice that the closure of $\exp (tJ)$ in $\SO(\ag)$ is a flat torus $T$ which still commutes with $\sC(J)$, and moreover, for any point $X \in \ag$, the orbit $TX$ lies in the orbit $\sC(J) X$. Then the action of the compact group $\sC'(J)$ generated by $\sC(J)$ and $T$ is orbit-equivalent to the action of $\sC(J)$ and has the same regular points. Moreover, for the former action the normal slice representation at the regular points is trivial. As $\overline{\sC}(J)$ is a subgroup of $\sC'(J)$, its stationary subalgebra at any point is a subalgebra of the stationary subalgebra of the group $\overline{\sC}(J)$, and the claim follows \footnote{We are thankful to Ramiro Lafuente for discussing this argument with us.}.
\end{observation*}

We prove the following.

\begin{proposition} \label{p:dimbound}
  In the assumptions of Proposition~\ref{p:irst}, there is a finite number of possible cases for the dimensions of $\ag$ and $\cV$, the algebras $\cP(\cV)$ and $\cC(\cV)$, and the representations of $\cP(\cV)$ on $\cV$ and on $\ag$. They are listed in Table~\ref{tab:dims1} below (where in the last two columns, $\br^n$ and $\bc^n$ mean the standard representations, $\Delta_n$ the spin representation and $\t$ a trivial representation).

{\renewcommand{\arraystretch}{1.4}
  \begin{table}[h!]
    \caption{Possible cases for $\ag, \cV, \cP(\cV)$ and $\cC(\cV)$.}
    \centering
    \begin{tabular}{|c|c|c|c|c|c|c|}
      \hline
       \#  & $\dim \ag$ & $\dim \cV$ & $\cP(\cV)$ & $\cC(\cV)$ & $\cP(\cV)$ \emph{on} $\cV$ & $\cP(\cV)$ \emph{on} $\ag$ \\
      \hline
      $1$ & $256$ & $15$ & $\so(15)$ & $\so(2)$ & $\br^{15}$ & $2 \Delta_{15}$\\
      \hline
      $2$ & $128$ & $15$ & $\so(15)$ & $0$ & $\br^{15}$ & $\Delta_{15}$\\
      \hline
      $3$ & $128$ & $15,14$ & $\so(14)$ & $0, \so(2)$ & $\br^{14} \oplus \t$ & $\Delta_{14}$\\
      \hline
      $4$ & $128$ & $15, 14, 13$ & $\so(13)$ & $0, \so(2), \spg(1)$ & $\br^{13} \oplus \t$ & $\Delta_{13}$\\
      \hline
      $5$ & $64$ & $11$ & $\so(11)$ & $0, \so(2), \spg(1)$ & $\br^{11}$ & $\Delta_{11}$\\
      \hline
      $6$ & $64$ & $11,10$ & $\so(10)$ & $\spg(1)$ & $\br^{10} \oplus \t$ & $2\Delta_{10}$\\
      \hline
      $7$ & $64$ & $9$ & $\so(9)$ & $\spg(1)$ & $\br^{9}$ & $4\Delta_{9}$\\
      \hline
      $8$ & $32$ & $9$ & $\so(9)$ & $\so(2)$ & $\br^{9}$ & $2\Delta_{9}$\\
      \hline
      $9$ & $32$ & $9,8$ & $\so(8)$ & $\spg(1)$ & $\br^{8} \oplus \t$ & $4\Delta^{\pm}_{8}, 4 \br^8$\\
      \hline
      $10$  & $32$ & $7$ & $\so(7)$ & $\spg(1)$ & $\br^{7}$ & $4\Delta_{7}$\\
      \hline
      $11$ & $16$ & $8$ & $\so(8)$ & $\so(2)$ & $\br^{8}$ & $2\Delta^{\pm}_{8}, 2 \br^8$\\
      \hline
      $12$ & $16$ & $8,7$ & $\so(7)$ & $\so(2)$ & $\br^{7} \oplus \t$ & $2\Delta_{7}$\\
      \hline
      $13$ & $16$ & $8$ & $\so(7)$ & $\so(2)$ & $\Delta_{7}$ & $2\Delta_{7}$\\
      \hline
      $14$ & $16$ & $8,7,6$ & $\so(6)$ & $\spg(1)$ & $\br^6 \oplus \t$ & $2\Delta_{6}$\\
      \hline
      $15$ & $16$  & $8$ & $\so(6)$ & $\spg(1)$ & $\bc^4$ \emph{(as $\su(4)$)} & $2\Delta_{6}$  \\
      \hline
      $16$ & $16$  & $8$ & $\so(6) \oplus \so(2)$ & $\spg(1)$ & $\br^6 \oplus \br^2$ & $\Delta_{6} \hat\otimes \br^2_{\so(2)}$  \\
      \hline
      $17$ & $16$  & $8$ & $\so(6) \oplus \so(2)$ & $\spg(1)$ & $\bc^4$ \emph{(as $\ug(4)$)} & $\Delta_{6} \hat\otimes \br^2_{\so(2)}$  \\
      \hline
      $18$ & $8$ & $7,6,5,4$ &  \multicolumn{4}{c|}{$\dim \cC(J) \ge 6$} \\
      \hline
    \end{tabular}
    \label{tab:dims1}
    \end{table}
}
\end{proposition}
\begin{proof}
  Recall that the assumptions of Proposition~\ref{p:irst} imply that $\dim \ag \ge 8$.

  Choose and fix $X \in \ag$ which is regular for the action of $\overline{\sC}(J)$. Let $\overline{\sC}(J)_X$ be the stabilizer subgroup of $X$, and $\overline{\cC}(J)_X \subset \overline{\cC}(J)$ its Lie algebra. By the GO property~\eqref{eq:NXJX}, the subalgebra $\overline{\cC}(J)_X$ is nontrivial and is not contained in $\cC(J)$, and $J \notin \overline{\cC}(J)_X$, as $J$ is non-singular. It follows that $\overline{\cC}(J)_X = \Span(N-J, N_1, \dots, N_k),\; k = \dim \overline{\cC}(J)_X - 1$, where $N \ne 0$ and $N, N_1, \dots, N_k \in \cC(J)$.

  Let $L_X = \cap \{\Ker \overline{N} \, | \, \overline{N} \in \overline{\cC}(J)_X\}$. As the orbit $\overline{\sC}(J) X$ is principal, the group $\overline{\sC}(J)_X$ acts trivially on the normal space $\nu_X (\overline{\sC}(J)X)$ (a neighbourhood of $X$ in which is a slice); recall that all the orbits of the action of $\overline{\sC}(J)$ are compact and coincide with the orbits of the action of the compact group $\sC(J)$. It follows that $L_X \supset \nu_X (\overline{\sC}(J)X)$. Furthermore, let $\h_X \subset \overline{\cC}(J)$ be the (maximal) subalgebra which centralises the subalgebra $\overline{\cC}(J)_X$ and is orthogonal to it relative to the restriction of the Killing form of $\so(\ag)$ to $\overline{\cC}(J)$. Note that the subspace $\h_X X$ is tangent to the orbit $\overline{\sC}(J) X$ at $X$ and is of dimension $\dim \h_X$, as the projection from the orthogonal complement to $\overline{\cC}(J)_X$ in $\overline{\cC}(J)$ to the tangent space $T_X (\overline{\sC}(J) X)$ is a linear bijection. Denote $d_X=\dim L_X$ (note that $d_X$ is locally constant, at least for regular points close to $X$). We obtain
  \begin{equation}\label{eq:dimLX}
    d_X \ge \dim \nu_X (\overline{\sC}(J)X) + \dim \h_X = \dim \ag - \dim \cC(J) - 1 + \dim \overline{\cC}(J)_X  + \dim \h_X.
  \end{equation}

  {
  \begin{lemma} \label{l:ineq}
    We have $d_X \le \frac34 \dim \ag$. Moreover, $\dim \overline{\cC}(J)_X  + \dim \h_X \ge \rk \cC(J)+1$ if $\overline{\cC}(J)_X$ is abelian, and $\dim \overline{\cC}(J)_X  + \dim \h_X \ge 5$ if it is not. Furthermore, if $d_X > \frac12 \dim \ag$, then the following holds.
    \begin{enumerate}[label=\emph{(\roman*)},ref=\roman*]
      \item \label{it:cldim1}
      $\dim \overline{\cC}(J)_X = 1$ and $\dim \h_X \ge \rk \cC(J)$.

      \item \label{it:cltrans}
      The group $\sP$ acts transitively on the unit sphere of $\cV$ (recall that this action is almost effective).

      \item  \label{it:clrk}
      $\rk \cC_P(J) = \rk \cP(\cV)$.
    \end{enumerate}
  \end{lemma}
  \begin{proof}
    Suppose $d_X > \frac34 \dim \ag$ for some regular point $X \in \ag$. Then there exists an element $N-J \in \overline{\cC}(J)_X$ whose kernel has dimension greater than $\frac34 \dim \ag$, where $N \in \cC(J) \subset \cN(\cV)$ is nonzero. As $N - J \ne 0$ (since $J \notin \cN(\cV)$), there is a nearby regular point $X' \in \ag$ such that $(N-J)X' \ne 0$, and an element  $N' \in \cN(\cV)$ such that $\dim \Ker (N'-J) > \frac 34 \dim \ag$. It follows that $\dim \Ker (N-N') \ge \dim \Ker (N'-J) + \dim \Ker (N'-J) - \dim \ag > \frac12 \dim \ag$. As $N-N' \ne 0$, we arrive at a contradiction with Lemma~\ref{l:dimkerN}.

    To prove the inequalities for $\dim \overline{\cC}(J)_X  + \dim \h_X$ we recall that $\overline{\cC}(J)_X = \Span(N-J, N_1, \dots, N_k),\; k = \dim \overline{\cC}(J)_X - 1$, where $N \ne 0$ and $N, N_1, \dots, N_k \in \cC(J)$. Note that $\kg = \Span(N_1, \dots, N_k) = \overline{\cC}(J)_X \cap \cC(J)$, and so $\kg$ is a subalgebra of $\overline{\cC}(J)_X$ of codimension one. As all algebras are reductive, we can assume that $N$ is chosen in such a way that $[N-J, \kg]= 0$ and $N-J \perp \kg$ (relative to the Killing form of $\so(\ag)$). Now if $\overline{\cC}(J)_X$ is abelian, we obtain $\dim \overline{\cC}(J)_X  + \dim \h_X \ge \rk \overline{\cC}(J)_X  + \rk \h_X = \rk \overline{\cC}(J)$, as $\h_X$ is the orthogonal complement to $\overline{\cC}(J)_X$ in the centralizer of $\overline{\cC}(J)_X$ in $\overline{\cC}(J)$. The required inequality follows as $\overline{\cC}(J) = \cC(J) \oplus \br J$. If $\overline{\cC}(J)_X$ is not abelian, then there exists a semisimple subalgebra $\kg' \subset \kg$, and so $\overline{\cC}(J)_X = \br \, (N-J) \oplus \kg \supset \br \, (N-J) \oplus \kg'$. Now if $\kg' \ne \kg$ we get $\overline{\cC}(J)_X \ge 5$, and if $\kg' = \kg$, then $\h_X$ contains the orthogonal complement $\br J'$ to the element $N-J$ in $\Span(N,J)$. Indeed, $[J', \overline{\cC}(J)_X]=0$, as both $J$ and $J-N$ commute with $\kg$, and $J' \perp \kg$, as $\kg$ is semisimple. It follows that $\dim \overline{\cC}(J)_X  + \dim \h_X \ge 4+1=5$.

    Now suppose that $d_X > \frac12 \dim \ag$, for some regular point $X \in \ag$.

    If $\dim \overline{\cC}(J)_X > 1$, then $\overline{\cC}(J)_X$ contains a nonzero element $N \in \cC(J) \subset \cN(\cV)$. Then $\dim \Ker N > \frac12 \dim \ag$ contradicting Lemma~\ref{l:dimkerN}. This proves assertion~\eqref{it:cldim1}. Hence we have $\overline{\cC}(J)_X = \br \, (N-J)$, for some nonzero $N \in \cC(J) \subset \cN(\cV)$. As $\overline{\cC}(J)_X$ is abelian, the inequality $\dim \h_X \ge \rk \cC(J)$ follows from the previous assertion.

    For assertion~\eqref{it:cltrans}, suppose a principal orbit of $\sP$ on the unit sphere of $\cV$ has positive codimension. Then there exists $J' \in \cV$ such that $J$ and $J'$ are linearly independent and $\cC_P(J) = \cC_P(J')$. It follows that $\cC(J) = \cC(J')$ (as $\cC(\cV)$ acts trivially on $\cV$), and so the $2$-dimensional subspace $\Span(J,J')$ is GO of centralizer type (see Lemma~\ref{l:W} for a more general statement). But from Proposition~\ref{p:ct} it follows that if $\Span(J,J')$ is non-singular, then the matrix $[J,J']$ is also non-singular. Now, we have an element $N \in \cC(J) = \cC(J')$ such that $\dim \Ker (N-J) > \frac12 \dim \ag$, which implies that $[J,J']=[J',N-J]$ is singular, a contradiction.

    To prove assertion~\eqref{it:clrk}, suppose that $\rk \cC_P(J) < \rk \cP(\cV)$ for some generic $J \in \cV$. As before, we have $\dim \Ker (N-J) > \frac12 \dim \ag$, for some $N \in  \cC(J)$. Recall that $\cC(J)=\cC_P(J) \oplus \cC(\cV)$ (direct sum of ideals), and so $N=N_1+N_2$, where $N_1 \in \cC_P(J)$ and $N_2 \in \cC(\cV)$. As $\rk \cC_P(J) < \rk \cP(\cV)$, there exists a nonzero $N' \in \cP(\cV)$ which commutes with $N_1$, but does not belong to $\cC_P(J)$. Then $[N',N-J] = -[N',J]$, which is a nonzero element of $\cV$ (since $N' \notin \cC_P(J)$), and hence is non-singular. This is a contradiction, as  $[N',N-J]$ is singular, since $\dim \Ker (N-J) > \frac12 \dim \ag$.
  \end{proof}
  }

  We separately consider two cases: $d_X > \frac12 \dim \ag$ and $d_X \le \frac12 \dim \ag$.

  {
  \begin{lemma} \label{l:dX>halfa}
  Suppose that for some generic $J \in \cV$ and regular $X \in \ag$ we have $d_X > \frac12 \dim \ag$. Then the only possible cases for $\dim \ag, \dim \cV, \cP(\cV), \cC(\cV)$ and the representations of $\cP(\cV)$ on $\cV$ and on $\ag$ are as given in cases $1, 2, 4, 5, 7, 8, 10, 12$ and $18$ of Table~\ref{tab:dims1}.
  \end{lemma}
  \begin{proof}
  From Lemma~\ref{l:ineq} we know that $d_X \le \frac34 \dim \ag$ and that $\dim \overline{\cC}(J)_X = 1$ and $\dim \h_X \ge \rk \cC(J)$, and so~\eqref{eq:dimLX} gives $\dim \cC(J) - \rk \cC(J) \ge \frac14 \dim \ag$. Furthermore, we have $\cC(J)= \cC_P(J) \oplus \cC(\cV)$, where $\cC(\cV)$ is isomorphic to a subalgebra of $\spg(1)$ (by Lemma~\ref{l:red}\eqref{it:redCdiva}), which gives
  \begin{equation}\label{eq:ccpjdX34}
  \dim \cC_P(J) - \rk \cC_P(J) \ge \tfrac14 \dim \ag - 2,
  \end{equation}
  where the right-hand side must be replaced with $\frac14 \dim \ag$ unless $\cC(\cV) \cong \spg(1)$.

  Furthermore, from Lemma~\ref{l:ineq}\eqref{it:cltrans},\eqref{it:clrk} we find that there are only two possibilities for the group $\sP$ (e.g., from \cite[Table~1]{Str}): either $\sP=G_2$ and then $\sC_P(J) = \SU(3)$ and $\dim \cV = 7$, or $\sP=\SO(2m+1), \, m \ge 1$, with its standard action on $\cV = \br^{2m+1}$ and with $\sC_P(J) = \SO(2m)$. In the first case, inequality~\eqref{eq:ccpjdX34} gives $\dim \ag \le 32$, and from formula~\eqref{eq:RH} we get $8 \mid  \dim \ag$, as $\dim \cV = 7$, so that $\dim \ag \in \{8,16,24,32\}$. But by Lemma~\eqref{l:redst}\eqref{it:redPir}, $\dim \ag$ must be divisible by the dimension of a nontrivial representation of $\cP(\cV)$, and the only representations of $\g_2$ of dimension less than or equal to $32$ have dimensions $7, 14$ and $27$, a contradiction. Thus $\cV= \br^{2m+1}, \, m \ge 1$, and $\cP(\cV) = \so(2m+1)$ (with its standard representation on $\cV$). Then $\cC_P(J) = \so(2m)$, and from inequality~\eqref{eq:ccpjdX34} we obtain
  \begin{equation}\label{eq:2mm-1}
  2m(m-1) \ge \tfrac14 \dim \ag - 2.
  \end{equation}
  On the other hand, by~\eqref{eq:RH}, $\dim \ag$ must be large enough for $\so(\ag)$ to admit a $(2m+1)$-dimensional non-singular subspace $\cV \subset \so(\ag)$.

  The rest of the argument follows the same scheme. First, from inequalities~\eqref{eq:RH} and~\eqref{eq:2mm-1} we determine all the possible values for $\dim \ag$ and $\dim \cV$. Then we apply Lemma~\eqref{l:redst}\eqref{it:redPir}: we know that $\cP(\cV) = \so(\cV)$, and that the $\cP(\cV)$-module $\ag$ is the sum of isomorphic irreducible (nontrivial) submodules (and we can have either one, or two, or four summands depending on their type and on $\cC(\cV)$). In particular, the dimension of these modules must divide $\dim \ag$, which in the most cases will be a power of $2$. Then we look at the dimensions of irreducible representation of the algebras $\so(2m+1)$ and list those of them which satisfy this condition; in the most cases, these will be the spin representations (see Table~\ref{tab:spin}). Possible candidates for $\cC(\cV)$ are also determined by Lemma~\eqref{l:redst}\eqref{it:redPir}.

  Following this procedure we first note that the number of cases will be finite: if $\dim \ag = 2^r u$, with $u$ odd, then by~\eqref{eq:RH} we get $\dim \cV \le 2r+1$, and so $m \le r$ giving $r(r-1) \ge 2^{r-3}u-1$ by~\eqref{eq:2mm-1}. It is not hard to see that already if $512 \mid \dim \ag$, the inequalities~\eqref{eq:RH} and~\eqref{eq:2mm-1} are incompatible. We now take $\dim \ag = 2^r u$, with $u$ odd, and with $2 \le r \le 8$ (if $r \le 1$, then $\dim \cV \le 2$ and hence $\cV$ is of centralizer type); we also note that for $r=2$ we must have $u \ge 3$, as $\dim \ag \ge 8$.

  Suppose $\dim \ag = 256 u$, for some odd $u$. From~\eqref{eq:RH} we obtain $2m+1 \le 16$, and so $m \le 7$. Then from~\eqref{eq:2mm-1} we get $u=1$ (and so $\dim \ag =256$) and $m=7$. Then $\dim \cV = 15$ and $\cP(\cV) = \so(15)$. By Lemma~\eqref{l:redst}\eqref{it:redPir}, the representation of $\cP(\cV)$ is the direct sum of isomorphic irreducible (nontrivial) modules, and so $\dim \ag = 256$ must be divisible by the dimension of these modules. This is only possible when $\ag = 2 \Delta_{15}$, and as $\Delta_{15}$ is of real type (see Table~\ref{tab:spin}), Lemma~\eqref{l:redst}\eqref{it:redPir} gives $\cC(\cV) \cong \so(2)$. We get case $1$ from Table~\ref{tab:dims1}.

  Next suppose $\dim \ag = 128 u$, for some odd $u$. Then from~\eqref{eq:RH} and~\eqref{eq:2mm-1} we obtain $\dim \ag = 128$ and $\dim \cV \in \{11, 13, 15\}$. From Lemma~\eqref{l:redst}\eqref{it:redPir}, the representation of $\cP(\cV) = \so(\cV)$ on $\ag$ is the direct sum of isomorphic irreducible modules, whose dimension divides $128$. For $\dim \cV = 15$, the only possible such representation is $\Delta_{15}$. As it is of real type (see Table~\ref{tab:spin}) we obtain $\cC(\cV) = 0$. Similarly, for $\dim \cV = 13$, the only possibility is $\ag=\Delta_{13}$. This module is of quaternionic type, and so $\cC(\cV)$ can be an arbitrary subalgebra of $\spg(1)$. If $\dim \cV = 11$, we must have $\ag = 2 \Delta_{11}$. But this contradicts Lemma~\eqref{l:redst}\eqref{it:redPir}, as $\Delta_{11}$ is of quaternionic type (see Table~\ref{tab:spin}). We obtain cases $2$ and $4$ (with $\dim \cV =13$) in Table~\ref{tab:dims1}.

  Suppose $\dim \ag = 64 u$, where $u$ is odd. Inequalities~\eqref{eq:RH} and~\eqref{eq:2mm-1} give $\dim \ag = 64$ and $\dim \cV \in \{9, 11\}$. Arguing as above we obtain that for $\dim \cV = 11$ we have $\ag = \Delta_{11}$ (then $\cC(\cV) \subset \spg(1)$ is arbitrary, as $\Delta_{11}$ is of quaternionic type). For $\dim \cV = 9$ we get $\ag = 4\Delta_{9}$, and then $\cC(\cV) \cong \spg(1)$ by Lemma~\eqref{l:redst}\eqref{it:redPir} and from Table~\ref{tab:spin}. This gives cases $5$ and $7$ in Table~\ref{tab:dims1}.

  When $\dim \ag = 32 u$, where $u$ is odd, inequalities~\eqref{eq:RH} and~\eqref{eq:2mm-1} are satisfied when $u=3$ and $m=4$ and when $u=1$ and $m = 3, 4$. But in the first case, we have a representation of $\so(9)$ on $\br^{96}$, which must be $6 \Delta_9$ contradicting Lemma~\eqref{l:redst}\eqref{it:redPir}. In the remaining two case we get $\dim \ag =32$ and either $\dim \cV = 9$, and then $\ag = 2 \Delta_9$ and $\cC(\cV)=\so(2)$ (as $\Delta_9$ is of real type from Table~\ref{tab:spin}), or $\dim \cV = 7$, and then $\ag = 4 \Delta_7$ and $\cC(\cV)=\spg(1)$. We arrive at cases $8$ and $10$ in Table~\ref{tab:dims1}.

  Suppose $\dim \ag = 16u$, where $u$ is odd. Inequalities~\eqref{eq:RH} and~\eqref{eq:2mm-1} give $m \in \{2,3\}$. If $m=2$, then $u=1$, and so $\dim \ag = 16$ and $\cP(\cV) = \so(5)$. Its only irreducible representation whose dimension divides $16$ is $\Delta_5$ (there is a complex irreducible representation of dimension $16$, but it gives a real representation of dimension $32$). But then $\ag = 2 \Delta_5$, in contradiction with Lemma~\eqref{l:redst}\eqref{it:redPir}, as $\Delta_5$ is of quaternionic type. Therefore we must have $m=3$ and then $\dim \cV = 7$ and $\cP(\cV) = \so(7)$. From~\eqref{eq:2mm-1} we obtain $u \in \{1,3\}$. If $u = 1$, we get $\dim \ag = 16$ and so $\ag = 2 \Delta_7$ and $\cC(\cV) \cong \so(2)$ by Lemma~\eqref{l:redst}\eqref{it:redPir}, as $\Delta_7$ is of real type.  We arrive at case $12$ in Table~\ref{tab:dims1}. If $u=3$, we get $\dim \ag = 48$. There are only two real, irreducible, nontrivial representations of $\so(7)$ whose dimension divides $48$: $\Delta_7$ of dimension $8$ and $\pi = R(\phi_1 + \phi_3)$ of dimension $48$. But we cannot have $\ag=6\Delta_7$ by Lemma~\eqref{l:redst}\eqref{it:redPir}. To see that the second case is also not possible, we note that $\pi$ is explicitly given as follows. Consider the space $\br^{56}$ with the basis $J_i \otimes e_a$, where $J_i, \, i=1, \dots, 7$, are anticommuting complex structure in $\br^8$, and $e_a, \, a=1, \dots, 8$, are orthonormal vectors in $\br^8$. The action of $\so(7)$ turns $\br^{56}$ into the $\so(7)$-module $\br^7 \otimes \Delta_7$. The map $\psi: \br^7 \otimes \Delta_7 \to \Delta_7$ defined by $\psi(J_i \otimes e_a)=J_ie_a$ is a module homomorphism. Then $\pi = \Ker \psi$. The representation $\pi$ is of real type, and so $\cC(\cV)=0$. In~\eqref{eq:dimLX}, we have $\cC(J) = \so(6)$ (as the action of $\so(7)$ on $\cV=\br^7$ is standard), $\dim \overline{\cC}(J)_X = 1$ and $\dim \h_X \ge 3$ by Lemma~\ref{l:ineq}\eqref{it:cldim1}, and so we get $d_X \ge 36$. On the other hand, from Lemma~\ref{l:ineq} we have $d_X \le 36$. It follows that for a generic $J \in \cV$ and a regular $X \in \ag$, there exists $N \in \cC(J) \subset \cN(\cV)$ such that $\dim \Ker (N-J) = 36$. Choosing $X' \in \ag \setminus \Ker(N-J)$ to be regular (with the same choice of $J$; recall that $J \notin \cN(\cV)$, so $N-J \ne 0$) we get a different element $N' \in \cC(J) \subset \cN(\cV)$ such that $\dim \Ker (N'-J) = 36$. Then $\dim \Ker (N'-N) \ge 24$, and hence $\dim \Ker (N'-N) = 24$ by Lemma~\ref{l:dimkerN}. But it turns out that no nonzero element of $\cN(\cV)$ can have such a large kernel. To see this, we look at the weights of $\pi$. Explicitly, take an arbitrary $w = \a_1 e_1 \wedge e_2 + \a_2 e_3 \wedge e_4  + \a_3 e_5 \wedge e_6 \in \so(7)$ (we can choose $w$ in the Cartan subalgebra $\Span(e_1 \wedge e_2,e_3 \wedge e_4,e_5 \wedge e_6) \subset \so(7)$) and act by $N=\pi(w)$ on an arbitrary element $X=\sum_{i=1}^{7} J_i \otimes u_i \in \ag$ (where $u_i \in \br ^8$ satisfy $\sum_{i=1}^{7} J_i u_i = 0$). Denote $K=\Delta_7(w) = \a_1 J_1 J_2 + \a_2 J_3 J_4  + \a_3 J_5 J_6 \in \so(8)$. We obtain $N.X= \sum_{i=1}^{7} J_i \otimes (Ku_i) + \sum_{s=1}^{3} 2\a_s (J_{2s} \otimes u_{2s-1} - J_{2s-1} \otimes u_{2s})$, and so $X \in \Ker N$ when $K u_{2s-1} = 2 \a_s u_{2s}$ and $K u_{2s} = -2 \a_s u_{2s-1}$ for $s=1,2,3$, and $Ku_7=0$. Note that $K^2$ has eigenvalues $-(\a_1 \pm \a_2 \pm \a_3)^2$, with the corresponding eigenspaces of multiplicity $2$ each. An easy combinatorial argument shows that the kernel of $N$ is of maximal dimension when (up to permutation and scaling) we have $\a_1=\a_2 = 1, \, \a_3 = 0$. Then $K^2$ has two eigenspaces, $E_0$ and $E_{-4}$, of dimension $4$ each, and so $X \in \Ker$ if and only if $u_5,u_6,u_7 \in E_0$ and $u_1, u_3 \in E_{-4}$ are arbitrary (and then $u_2, u_4 \in E_{-4}$ are uniquely determined). This subspace has dimension $20$, and then we subtract $4$, the dimension of $\Ker K$ (note that $K = \psi(\pi(w)) = \psi(N)$), to get $\dim \Ker \pi(w) = 16$. Thus $\dim \Ker N \le 16 < 24$, for all nonzero $N \in \cN(\cV)$, a contradiction.

  Let now $\dim \ag = 8u$, where $u$ is odd. Inequalities~\eqref{eq:RH} and~\eqref{eq:2mm-1} give $m \le 3$ and $u \le 7$. If $u=7$ we get $\dim \ag = 56$ and $\dim \cV = 7$, and so either $\ag=8 \br^7$ or $\ag=7 \Delta_7$, both contradicting Lemma~\ref{l:redst}\eqref{it:redPir}. If $u=5$ we obtain $\dim \ag = 40$ and $\dim \cV = 7$ which gives $\ag=5 \Delta_7$, still in contradiction with Lemma~\ref{l:redst}\eqref{it:redPir}. For $u=3$, we have $\dim \ag = 24$ and either $\dim \cV = 7$ which gives $\ag=3 \Delta_7$, or $\dim \cV = 7$ which gives $\ag=3 \Delta_5$. Both cases contradict Lemma~\ref{l:redst}\eqref{it:redPir}. Finally, let $u=1$, and so $\dim \ag = 8$. If $m \in \{2,3\}$, then $\dim \cV = 2m +1 \ge 5$ and $\cP(\cV) = \so(2m+1)$, so that $\cC(J) \supset \so(2m)$ and hence $\dim \cC(J) \ge 6$; this agrees with the inequalities in case $18$ of Table~\ref{tab:dims1}. If $m=1$, we have $\dim \cV = 3$ and $\cP(\cV)=\so(3)$. Then the representation of $\cP(\cV)$ is either $2 \Delta_3$ or is the irreducible $8$-dimensional representation $\pi$ of $\so(3)$. But the first case contradicts Lemma~\ref{l:redst}\eqref{it:redPir}, as $\Delta_3$ is of quaternionic type, and in the second case, $\Lambda^2 \pi$ contains only one nontrivial $3$-dimensional $\so(3)$-submodule, which is $\pi(\so(3))$ itself, and so $\cV = \cP(\cV)$, in contradiction with the fact that $\cV$ is of common type.

  And finally, if $\dim \ag = 4u$, where $u$ is odd, we must have $u \ge 3$ as $\dim \ag \ge 8$, and $m=1$ from~\eqref{eq:RH}, in contradiction with~\eqref{eq:2mm-1}.
  \end{proof}
  }

  Now we suppose that $d_X \le \frac12 \dim \ag$. Then~\eqref{eq:dimLX} gives us a stronger inequality, but we do not get any extra information on $\cP(\cV)$ (like in Lemma~\ref{l:ineq} for the case $d_X > \frac12 \dim \ag$).

  {
  \begin{lemma} \label{l:dXlehalfa}
  Suppose that for some generic $J \in \cV$ and regular $X \in \ag$ we have $d_X \le \frac12 \dim \ag$. Then the only possible cases for $\dim \ag, \dim \cV, \cP(\cV), \cC(\cV)$ and the representations of $\cP(\cV)$ on $\cV$ and on $\ag$ are as given in cases $2, 3, 4, 5, 6, 8, 9$  and $11$--$18$ of Table~\ref{tab:dims1}.
  \end{lemma}
  \begin{proof}
  From Lemma~\ref{l:ineq} we know that $\dim \overline{\cC}(J)_X  + \dim \h_X \ge \min(5, \rk \cC(J)+1)$, and so inequality~\eqref{eq:dimLX} gives
  \begin{equation}\label{eq:ccpjdX12}
  \begin{split}
     \dim \cC(J) \ge \tfrac12 \dim \ag + \min(4, \rk \cC(J)), &  \quad \text{and so}   \\
     \dim \cC_P(J) \ge \tfrac12 \dim \ag + \min(4, \rk \cC_P(J)), & \quad \text{when } \cC(\cV) = 0, \\
     \dim \cC_P(J) \ge \tfrac12 \dim \ag + \min(3, \rk \cC_P(J)), & \quad \text{when } \cC(\cV) \cong \so(2), \\
     \dim \cC_P(J) \ge \tfrac12 \dim \ag + \min(1, \rk \cC_P(J)-2),  & \quad \text{when } \cC(\cV) \cong \spg(1),
  \end{split}
  \end{equation}
  where the latter three inequalities follow from the fact that $\cC(J)=\cC_P(J) \oplus \cC(\cV)$, and $\cC(\cV)$ is either trivial, or is isomorphic to $\so(2)$ or $\spg(1)$.

  As $\cC_P(J)$ is a subalgebra of $\cP(\cV)$ (the latter acts faithfully on $\cV$) which is the isotropy subalgebra of a generic element $J \in \cV$, we have $\cC_P(J) \subset \so(\cV \cap J^\perp)$, and so $\dim \cC_P(J) \le \frac12(\dim \cV - 1)(\dim \cV - 2)$. On the other hand, from~\eqref{eq:RH} we know that if $\dim \ag = 2^r u$, with $u$ odd, then $\dim \cV \le 2r+1$, and so $\dim \cC_P(J) \le r(2 r- 1)$. Furthermore, the algebra $\cC_P(J)$ cannot be trivial (for otherwise $\cV$ is of centralizer type), and so the right-hand side of~\eqref{eq:ccpjdX12} is, at the very least, equal $\frac12 \dim \ag -1$. We get $r(2 r- 1) \ge 2^{r-1}u-1$ which implies that, similar to the case of Lemma~\ref{l:dX>halfa}, the number of possible cases is finite. It is not hard to see that already if $256 \mid \dim \ag$, we get a contradiction. We now take $\dim \ag = 2^r u$, with $u$ odd, and with $2 \le r \le 7$ (if $r \le 1$, then $\dim \cV \le 2$ and hence $\cV$ is of centralizer type); we also note that for $r=2$ we must have $u \ge 3$, as $\dim \ag \ge 8$ and, similarly to the proof of Lemma~\ref{l:dX>halfa}, check all the possible cases when both~\eqref{eq:ccpjdX12} and~\eqref{eq:RH} are satisfied. To do that, we first determine possible values for $\dim \cV$ coming from~\eqref{eq:RH}. Then we consider subalgebras of $\so(\cV)$ for which~\eqref{eq:ccpjdX12} is satisfied. We use the well known results on maximal subalgebras of $\so(n)$ \cite{Dyn} and also the list of representations of compact simple groups with nontrivial principal isotropy subgroups given in~\cite[Table~A]{HH}, to determine $\cC_P(J)$. This will give us a list of possible subalgebras $\cP(\cV)$ together with their representations on $\cV$. To determine the representations of $\cP(\cV)$ on $\ag$ we use Lemma~\ref{l:redst}\eqref{it:redPir}: $\ag$ must be the sum of either four, or two, or one isomorphic, irreducible $\cP(\cV)$ modules (depending on $\cC(\cV)$ and the type of spin representation given in Table~\ref{tab:spin}).

  Suppose $\dim \ag = 128 u$, where $u$ is odd. Then from~\eqref{eq:RH} we have $\dim \cV \le 15$. Then $u=1$ and so $\dim \ag = 128$, and from~\eqref{eq:ccpjdX12} we obtain $\dim \cC_P(J) \ge 63$. It follows that $\dim \cV \ge 13$ and that $\cP(\cV)$ is one of the algebras $\so(15), \so(14), \so(13)$ or $\so(13) \oplus \so(2)$. In the first three cases, the only representations whose dimensions divide $128$ are the corresponding spin representations, and then the possible candidates for $\cC(\cV)$ are determined from Table~\ref{tab:spin}. This gives cases $2,3$ and $4$ of Table~\ref{tab:dims1}. The case $\cP(\cV)=\so(13) \oplus \so(2)$ is not possible, as then from\eqref{eq:ccpjdX12} we must have $\cC(\cV) \cong \spg(1)$. But the representation of $\so(13)$ on $\ag$ must be $\Delta_{13}$ by Lemma~\ref{l:redst}\eqref{it:redPir}, and the centralizer of its image is $\spg(1)$ (see Table~\ref{tab:spin}), not $\so(2) \oplus \spg(1)$.

  Let $\dim \ag = 64 u$, where $u$ is odd. From~\eqref{eq:RH} we have $\dim \cV \le 11$. Then $u=1, \,\dim \ag = 128$ and $\dim \cC_P(J) \ge 31$. An argument similar to the above gives cases $5$ and $6$ of Table~\ref{tab:dims1}. Similarly, if $\dim \ag = 32 u$, for some odd $u$, we get $\dim \cV \le 9$ and $\dim \cC_P(J) \ge 16u-1$. Then $u=1,\, \dim \ag = 32$, and the only possibilities for $\cP(\cV)$ are $\so(8)$ and $\so(9)$. Using of Lemma~\ref{l:redst}\eqref{it:redPir}, we obtain cases $8$ and $9$ of Table~\ref{tab:dims1}.

  Now suppose that $\dim \ag = 16 u$, where $u$ is odd. From~\eqref{eq:RH} we obtain $\dim \cV \le 8$. Then $\cC_P(J) \subset \so(7)$, and so from~\eqref{eq:ccpjdX12} we get $u=1$ and $\dim \cC_P(J) \ge 8 + \min(1, \rk \cC_P(J)-2)$, in the worst possible case. Then from~\cite{Dyn} and~\cite[Table~A]{HH}, the list of possible candidates for $\cP(\cV) \subset \so(\cV)$ is as follows: $\so(8), \so(7), \mathfrak{spin}(7), \so(6), \so(6) \oplus \so(2), \ug(4)$ and $\su(4)$ given by their standard (or standard plus trivial) representations on $\cV$, with the corresponding principal isotropy subalgebras $\cC_P(J)$, respectively, $\so(7), \so(6), \g_2, \so(5), \so(5), \ug(3)$ and $\su(3)$, and also $\g_2$, with $\cC_P(J) = \su(3)$. In all these cases, except for the latter one, we determine the representation of $\cP(\cV)$ on $\ag$ using Lemma~\ref{l:redst}\eqref{it:redPir} and the type of the corresponding spin representations given in Table~\ref{tab:spin}; we obtain cases
  $11, 12, 13, 14, 16, 17$  and $15$ in Table~\ref{tab:dims1}, respectively. The fact that case $\cP(\cV) = \g_2$ is not possible also follows from Lemma~\ref{l:redst}\eqref{it:redPir}, as the only irreducible representations of $\g_2$ of dimension at most $16$ are in dimensions $7$ and $14$.

  Let now $\dim \ag = 8 u$, with $u$ odd. Then~\eqref{eq:RH} gives $\dim \cV \le 7$, and so from~\eqref{eq:ccpjdX12} we obtain $\dim \cC_P(J) \ge 4u -1$. As $\cC_P(J) \subset \so(6)$, we must have $u \in \{1,3\}$. But if $u=3$, the only possibility is $\cP(\cV)=\so(7)$, and so $\ag = 3 \Delta_7$, in contradiction with Lemma~\ref{l:redst}\eqref{it:redPir}. Then $u=1$, and the first inequality of~\eqref{eq:ccpjdX12} implies $\dim \cC(J) \ge 6$, as required in case $18$ of Table~\ref{tab:dims1}; the same inequality also implies $\dim \cV \ge 4$.

  To complete the proof it remains to consider the case $\dim \a = 4 u$, with $u$ odd. Note that $\dim \ag \ge 8$, and so $u \ge 3$, and by~\eqref{eq:RH}, $\dim \cV \le 3$. But then $\cP(\cV) \subset \so(3)$ and so $\dim \cC_P(J) \le 1$, in contradiction with~\eqref{eq:ccpjdX12}.
 \end{proof}
  }
This completes the proof of Proposition~\ref{p:dimbound}.
\end{proof}

The inequalities in the last row of Table~\ref{tab:dims1} are precisely those which we required in the proof of Lemma~\ref{l:dima8}. Hence Proposition~\ref{p:irst} is proved for $\dim \ag \le 8$.

\subsection{Proof of Proposition~\ref{p:irst} for \ensuremath{\texorpdfstring{\dim}{dim} \texorpdfstring{\ag}{\unichar{"1D51E}} > 8}}
\label{ss:cases}

In this section we complete the proof of Proposition~\ref{p:irst} by showing that in no cases with $\dim \ag > 8$ in Table~\ref{tab:dims1}, the subspace $\cV$ is of common type. Our arguments follow the same scheme. We consider the representation of $\cN(\cV) = \cP(\cV) \oplus \cC(\cV)$. In the most cases, this representation is the exterior tensor product of the corresponding spin representation of $\cP(\cV)$ and the representation of $\cC(\cV)$ on either $\br^2$ or $\br^4$ (recall that all the elements of $\cC(\cV)$ square to a multiple of the identity by Lemma~\ref{l:red}\eqref{it:redCdiva}). Using the exterior tensor product formula
\begin{equation}\label{eq:box}
  \Lambda^2(\pi \hat\otimes  \pi')=\Lambda^2(\pi)\hat\otimes S^2(\pi') + S^2(\pi) \hat\otimes \Lambda^2(\pi')
\end{equation}
 we show that in each case, either the $\cN(\cV)$-submodule $\cV \subset \Lambda^2(\ag)$ of the required dimension does not exist, or it is of Clifford type. In many cases we will use an explicit form of the spin representation constructed from anticommuting product structures or anticommuting complex structures on $\ag$.

\subsubsection{$\cP(\cV) = \so(15)$ or $\cP(\cV) = \so(14)$}
\label{sss:so1514}

From Table~\ref{tab:dims1}, case $1$, we have $\dim \ag =256$, and $\cN(\cV)$ is represented on $\ag$ by $\Delta_{15} \hat\otimes \br^2_{\so(2)}$. The image of $\Delta_{15}$ in $\Lambda^2(\br^{128})$ is given by $\Span(J_iJ_j \, | \, 1 \le i < j \le 15)$, where $J_1, \dots, J_{15}$ are anticommuting complex structures on $\br^{128}$. Moreover, $\Lambda^2(\Delta_{15})=\Lambda^1(\br^{15}) + \Lambda^2(\br^{15}) + \Lambda^5(\br^{15}) + \Lambda^6(\br^{15})$ and $S^2(\Delta_{15})=\Lambda^0(\br^{15}) + \Lambda^3(\br^{15}) + \Lambda^4(\br^{15}) + \Lambda^7(\br^{15})$, where the module $\Lambda^s(\br^{15}), \, s=0, \dots, 7$, is spanned by products of $s$ pairwise different $J_i$'s. Then from~\eqref{eq:box} we obtain that the only $15$-dimensional $(\Delta_{15} \hat\otimes \br^2_{\so(2)})$-submodule $\cV$ of $\Lambda^2(\ag)$ is $\Lambda^1(\br^{15}) \hat\otimes 1_{\so(2)}$ spanned by the matrices $\begin{smallmatrix} J_i & 0 \\ 0 & J_i \end{smallmatrix}$ relative to some orthonormal basis for $\ag$. Any linear combination of such matrices squares to a multiple of the identity, which is a contradiction, as $\cV$ is of Clifford type.

For case $2$ in Table~\ref{tab:dims1}, the contradiction follows from the fact that for a generic $J \in \cV$, the representation of the subalgebra $\cC(J)$ is irreducible (it is $\Delta_14$), and so $J$ must square to a multiple of the identity. The same argument also applies in case $3$.

\subsubsection{$\cP(\cV) = \so(13)$}
\label{sss:so13}

From Table~\ref{tab:dims1}, case $4$, we have $\dim \ag =128$ and $\cV=\cV' \oplus \t$, where $\cP(\cV)$ acts on $\cV' = \br^{13}$ by its standard representation and $t$ is a trivial module of dimension at most $2$. Then $\cV'$ is also an $\cN(\cV)$-module, and hence is a non-singular GO subspace. Now let $J_i \in \so(\ag), \, i=1, \dots, 15$, be pairwise anticommuting complex structures on $\ag$. Then the image of $\Delta_{13}$ in $\Lambda^2(\ag)$ is given by $\Span(J_iJ_j \, | \, 1 \le i < j \le 13)$, where $J_1, \dots, J_{15}$ are anticommuting complex structures on $\br^{128}$, and moreover, $\Lambda^2(\Delta_{13}) = 3 \cdot 1 + 3\Lambda^1(\br^{13})+ \Lambda^2(\br^{13})+ 3\Lambda^5(\br^{13}) + \Lambda^6(\br^{13}) + 3\Lambda^9(\br^{13}) + \Lambda^{10}(\br^{13})$, where the modules $\Lambda^s(\br^{13})$ are spanned by products of $s$ pairwise different $J_i, \, 1 \le i \le 13$. The module $3\Lambda^1(\br^{13})$ is given by $\Span_{i=1}^{13} (J_i) \oplus \Span_{i=1}^{13} (J_iJ_{14})\oplus \Span_{i=1}^{13} (J_iJ_{15})$, and so $\cV'=\Span_{i=1}^{13} (J_i(a\,\id+bJ_{14}+cJ_{15}))$, for some $a,b,c \in \br$, not all zeros. But then $J^2$ is a multiple of the identity, for any $J \in \cV'$. It follows that $\cV'$ is a GO subspace of Clifford type, which is a contradiction with~\cite{Rie} (see Theorem~\ref{tha:Cliff}), as $\dim \cV'=13$.

\subsubsection{$\cP(\cV) = \so(11)$}
\label{sss:so11}

From Table~\ref{tab:dims1}, case 5, we have $\dim \ag =64$, and the representation of $\cP(\cV) = \so(11)$ on $\ag$ is the spin representation $\Delta_{11}$. It is of quaternionic type, and so its image lies in $\spg(16) \subset \so(64)$. We have the following decomposition into irreducible $\so(11)$ modules: $\Lambda^2(\Delta_{11}) = 3\Lambda^0(\br^{11})+ \Lambda^1(\br^{11})+ \Lambda^2(\br^{11})+ \Lambda^5(\br^{11})+ 3\Lambda^3(\br^{11})+3\Lambda^4(\br^{11})$. Explicitly, we take $J_i \in \so(\ag), \, i=1, \dots, 11$, to be pairwise anticommuting complex structures on $\ag=\br^{64}$, and $\spg(1)=\Span(K_1, K_2,K_3)$ to be the centralizer of $\spg(16)$. Denote $\Phi_s=\Span(J_{i_1} \dots J_{i_s} \, | \, 1 \le i_1 < \dots < i_s \le 11)$. Then $\spg(1)$ is the trivial $\so(11)$-module, $\spg(16) = \Phi_1 + \Phi_2 + \Phi_5$ as an $\so(11)$-module, and $3\Lambda^s(\br^{11})= \Span_{a=1}^3 (K_a\Phi_s)$ for $s=3,4$. From Table~\ref{tab:dims1} we have $\dim \cV =11$, and so the only possibility for $\cV$ is $\cV=\Lambda^1(\br^{11})=\Phi_1=\Span_{i=1}^{11}(J_i)$. But then any element of $\cV$ squares to a multiple of the identity, a contradiction.

\subsubsection{$\cP(\cV) = \so(10)$}
\label{sss:so10}

From Table~\ref{tab:dims1}, case 6, we have $\dim \ag =64$ and $\cV=\cV' \oplus \t$, where $\cP(\cV)$ acts on $\cV' = \br^{10}$ by its standard representation and $t$ is a trivial module of dimension at most $1$. Then $\cV'$ is also an $\cN(\cV)$-module, and hence is also a non-singular GO subspace. Furthermore, $\ag$ is the sum of two isomorphic representations $\Delta_{10}$ on $\br^{32}$, which can be viewed as the restriction of the spin representation $\Delta_{11}$ to $\so(10)$. From the decomposition of $\Lambda^2(\Delta_{11})$ given in~\ref{sss:so11} we obtain that $\Lambda^2(\ag)$ contains the sum of only two $10$-dimensional $\so(10)$-submodules. Explicitly, if we take $J_i \in \so(\ag), \, i=1, \dots, 11$, to be pairwise anticommuting complex structures on $\ag=\br^{64}$, and take $\Span(J_i J_j \, | \, 1 \le i < j \le 10)$ for the image of $2\Delta_{10}$, then the sum of two standard $\so(10)$-submodules is given by $\Span(J_i \, | \, 1 \le i \le 10) \oplus \Span(J_iJ_{11} \, | \, 1 \le i \le 10)$. Hence $\cV'=\Span(J_i(a\,I_{64}+bJ_{11}) \, | \, 1 \le i \le 10)$, for some $a,b \in \br$, not both zeros. But then $J^2$ is a multiple of the identity, for any $J \in \cV'$, and so $\cV'$ is a GO subspace of Clifford type, in contradiction which~\cite{Rie} (see Theorem~\ref{tha:Cliff}), as $\dim \cV'=10$.

\subsubsection{$\cP(\cV) = \so(9)$}
\label{sss:so9}

We use the same approach in both cases $7$ and $8$ in Table~\ref{tab:dims1}. In both cases, $\dim \cV =9$.

In case $8$, we have $\dim \ag =32, \, \cN(\cV) = \so(9) \oplus \so(2)$ and $\ag = \Delta_9 \hat\otimes \br^2_{\so(2)}$ as an $\cN(\cV)$-module, where $\br^2_{\so(2)}$ is the standard $\so(2)$-module. We know that $\Lambda^2(\Delta_9)= \Lambda^2(\br^{9})+ \Lambda^3(\br^{9})$ and $S^2(\Delta_9)= 1+ \Lambda^1(\br^{9})+ \Lambda^4(\br^{9})$, and so from~\eqref{eq:box} we obtain that $\Lambda^2(\ag)$ contains only one $9$-dimensional $\cN(\cV)$-submodule $\Lambda^1(\br^{9}) \hat\otimes 1_{\so(2)}$. Explicitly, if $S_i \in \so(16), \, i=1, \dots, 9$, are pairwise anticommuting product structures on $\br^{16}$ (so that $S_i^t=S_i$ and $S_iS_j+S_jS_i = 2\,\id$ for $i \ne j$), we can choose a basis for $\ag$ in such a way that $\cP(\cV) = \Span\big(\left(\begin{smallmatrix}S_iS_j&0\\ 0&S_iS_j\end{smallmatrix}\right) \, | \, 1 \le i < j \le 9 \big)$ and $\cC(\cV) = \br \, \left(\begin{smallmatrix}0& I_{16} \\-I_{16}&0\end{smallmatrix}\right)$. Then $\cV = \Lambda^1(\br^{9}) \hat\otimes 1_{\so(2)} = \Span\big(\left(\begin{smallmatrix}0&S_i\\ -S_i&0\end{smallmatrix}\right) \, | \, 1 \le i \le 9 \big)$, and so every element of $\cV$ squares to a multiple of the identity, a contradiction.

In case $9$, we have $\dim \ag =64, \, \cN(\cV) = \so(9) \oplus \spg(1)$ and $\ag = \Delta_9 \hat\otimes \br^4_{\spg(1)}$ as an $\cN(\cV)$-module, where $\br^4_{\spg(1)}$ is the standard $\spg(1)$-module. Arguing as in the previous paragraph and noticing that $\Lambda^2(\br^4_{\spg(1)}) = 3 \cdot 1 + \br^3$ as the $\so(3)$ module, we obtain that $\cV$ is contained in the sum of three isomorphic $9$-dimensional standard $\cN(\cV)$-submodules. In the notation of the previous paragraph and of Section~\ref{ss:dima8}, we obtain that (relative to particular orthonormal bases for $\br^4$ and $\br^{16}$) there exists a unit, imaginary quaternion $a$ such that $\cV = \Span(R_a \otimes S_i \, | \, 1 \le i \le 9\}$. Then every element of $\cV$ squares to a multiple of the identity which again gives a contradiction.

\subsubsection{$\cP(\cV) = \so(8)$}
\label{sss:so8}

From Table~\ref{tab:dims1}, this is possible in two cases, $9$ and $11$. The argument in both cases is similar.

In case $11$, we have $\dim \ag = 16$. If the representation of $\cP(\cV)$ on $\ag$ is given by $2\Delta_8^{\pm}$, then the representation of $\cN(\cV)= \so(8) \oplus \so(2)$ on $\ag$ is $\Delta_8^{\pm} \hat\otimes \br^2_{\so(2)}$. But then for a generic $J \in \cV$, the representation of the subalgebra $\cC(J)= \so(8) \oplus \so(2)$ on $\ag$ is $\Delta_7^{\pm} \hat\otimes \br^2_{\so(2)}$, which is irreducible. Then $J^2$ is a multiple of the identity, a contradiction. Furthermore, if the representation of $\cP(\cV)$ on $\ag$ is given by $2\br^8$, then the representation of $\cN(\cV)= \so(8) \oplus \so(2)$ on $\ag$ is $\br^8{\pm} \hat\otimes \br^2_{\so(2)}$, and so by~\eqref{eq:box} we obtain $\Lambda^2(\br^8 \hat\otimes \br^2_{\so(2)}) = (1+S^2_0(\br^8)) \hat\otimes 1_{\so(2)} + \Lambda^2(\br^8) \hat\otimes (1 + \br^2)_{\so(2)}$, which contains no $8$-dimensional submodules, a contradiction.

In case $9$, we have $\dim \ag = 32$. If the representation $\cN(\cV)= \so(8) \oplus \spg(1)$ on $\ag$ is $\Delta_8^{\pm} \hat\otimes \br^4_{\spg(1)}$, then for a generic $J \in \cV$, the representation of the subalgebra $\cC(J)= \so(8) \oplus \spg(1)$ on $\ag$ is $\Delta_7^{\pm} \hat\otimes \br^2_{\so(2)}$. It is irreducible, and so $J^2$ is a multiple of the identity. If the representation of $\cN(\cV)= \so(8) \oplus \spg(1)$ on $\ag$ is $\br^8{\pm} \hat\otimes \br^4_{\spg(1)}$, then~\eqref{eq:box} gives $\Lambda^2(\br^8 \hat\otimes \br^4_{\spg(1)}) = (1+S^2_0(\br^8)) \hat\otimes (3 \cdot 1 + \br^3)_{\so(3)} + \Lambda^2(\br^8) \hat\otimes (1 + 3 \cdot \br^3)_{\so(3)}$, which contains no $8$-dimensional and no $9$-dimensional submodules, again a contradiction.

\subsubsection{$\cP(\cV) = \so(7)$}
\label{sss:so7}

Table~\ref{tab:dims1} gives three cases, $10, 12$ and $13$.

We first consider cases $12$ and $13$, when $\dim \ag = 16$. Then $\cN(\cV) = \so(7) \oplus \so(2)$ and $\ag = \Delta_7 \hat\otimes \br^2_{\so(2)}$ as an $\cN(\cV)$-module. As $\Lambda^2(\Delta_7) = \Lambda^1(\br^7) + \Lambda^2(\br^7)$ and $S^2(\Delta_7) = 1 + \Lambda^3(\br^7)$, we obtain $\Lambda^2(\ag) = (\Lambda^1(\br^7) + \Lambda^2(\br^7)) \hat\otimes (1 + \br^2)_{\so(2)} + (1 + \Lambda^3(\br^7)) \hat\otimes 1_{\so(2)}$. Note that this decomposition does not contain $\Delta_7 \hat\otimes (1_{\so(2)}$ (and so case $13$ cannot occur), and moreover, that there are only two possibilities for $\cV$: either $\cV= \br^7 \hat\otimes 1_{\so(2)}$ or $\cV= \br^7 \hat\otimes 1_{\so(2)} + 1_{\so(7) \oplus \so(2)}$. We can explicitly present the subalgebra $\cN(\cV) \subset \so(16)$ and the module $\cV$ as follows. Let $J_i, \, i = 1, \dots, 7$, be pairwise anticommuting complex structures on $\br^8$. Then $\cP(\cV) = \Span\big(\left(\begin{smallmatrix}J_iJ_j&0\\0&J_iJ_j\end{smallmatrix}\right) \, | \, 1 \le i < j \le 7 \big)$ and $\cC(\cV) = \br \, \left(\begin{smallmatrix}0& I_{8} \\-I_{8}&0\end{smallmatrix}\right)$ (relative to a choice of an orthonormal basis for $\ag = \br^{16}$. The module $\cV$ is given by either $\cV = \Span\big(\left(\begin{smallmatrix}J_i&0\\ 0&J_i \end{smallmatrix}\right) \, | \, 1 \le i \le 7 \big)$, or by $\cV = \Span\big(\left(\begin{smallmatrix}J_i&0\\0&J_i\end{smallmatrix}\right) \, | \, 1 \le i \le 7 \big) \oplus \cC(\cV)$. But in the former case, every element of $\cV$ squares to a multiple of the identity giving a contradiction, and in the latter case, any element $\left(\begin{smallmatrix}J_i&I_8\\-I_8&J_i\end{smallmatrix}\right) \in \cV$ is singular, which is, again, a contradiction.

In case $10$, we have $\dim \ag =32$. From Table~\ref{tab:dims1}, we obtain that $\cN(\cV)=\so(7) \oplus \spg(1)$ and $\ag=\Delta_7 \hat\otimes \br^4_{\spg(1)}$ as an $\cN(\cV)$-module. Then from~\eqref{eq:box}, $\Lambda^2(\ag) = (\Lambda^1(\br^7) + \Lambda^2(\br^7)) \hat\otimes (1 + 3 \cdot \br^3)_{\so(3)} + (1 + \Lambda^3(\br^7)) \hat\otimes (3 \cdot 1 + \br^3)_{\so(3)}$. Then $\cV = \br^7 \hat\otimes 1_{\so(3)}$ and is spanned
by the block-diagonal matrices $\diag(J_i, J_i, J_i, J_i), \, 1 \le i \le 7$, where $J_i$ are as in the previous paragraph, and so again, every element of $\cV$ squares to a multiple of the identity giving a contradiction.

\subsubsection{$\cP(\cV) = \so(6) \cong \su(4)$ or $\cP(\cV) = \so(6) \oplus \so(2) \cong \ug(4)$}
\label{sss:so6}

These are cases 14-17 in Table~\ref{tab:dims1}. In all the cases have $\dim \ag = 16$ and $\cC(\cV)= \spg(1)$. As $\Delta_6$ is of complex type (it is the standard representation of $\su(4)$ on $\br^8$), the centralizer of the image of $2\Delta_6$ in $\so(16)$ is isomorphic to $\ug(2)$. It splits into the direct sum of $\cC(\cV)=\spg(1)$ and the $1$-dimensional centre $\ug(1)$ which may or may not be a summand of $\cP(\cV)$ (in cases $16, 17$ and in cases $14, 15$, respectively). Denote $\cP'= \cP(\cV) + \ug(1) \cong \ug(4)$ (the sum may not be direct). Then both $\cV$ and $\cP'$ lie in $\spg(4) \subset \so(16)$, the centralizer of $\cC(\cV)$. As $\ug(4)$ has the same rank as $\spg(4)$, there is only one, up to conjugation, subalgebra of $\spg(4)$ isomorphic to $\ug(4)$. Then $(\spg(4),\ug(4))$ is a (Hermitian) symmetric pair, and so the $20$-dimensional $\ug(4)$-submodule of $\spg(4)$ which complements $\ug(4) \subset \spg(4)$ is $\ug(4)$-irreducible. It is also $\su(4)$-irreducible (and is isomorphic to $\Lambda^3(\br^6)$) which implies that there is no $\ug(4)$- and no $\su(4)$-submodules $\cV \subset \spg(4)$ of dimension in $\{6,7,8\}$, a contradiction.

\medskip

This completes the proof of Proposition~\ref{p:irst}.

\subsection{Reducible non-singular GO subspaces}
\label{ss:redu}

In this section we complete the proof of Theorem~\ref{th:class}. In view of Proposition~\ref{p:irst}, we can assume that $\ag$ is $\cV$-reducible (recall that the inner product $\ipr$ on $\cV$ is standard).

Suppose $\cV \subset \so(\ag)$ is a non-singular GO subspace such that $\ag$ is $\cV$-reducible. Let $\ag =\oplus_{i=1}^p \ag_i,\; p > 1$, be the orthogonal decomposition of $\ag$ into $\cV$-invariant subspaces such that no subspace $\ag_i$ contains a proper, nontrivial $\cV$-invariant subspace. In the notation of Section~\ref{ss:nc}, for $i =1, \dots, p$, denote $\pi_i: \so(\ag) \to \so(\ag_i)$ the orthogonal projection relative to the Killing form of $\so(\ag)$ and $\cV_i = \pi_i(\cV) \subset \so(\ag_i)$. From Lemma~\ref{l:redst}\eqref{it:redV} we know that $\cV_i \subset \so(\ag_i)$ is again a non-singular GO subspace. Moreover, $\ag_i$ is $\cV_i$-irreducible, and so the pair $(\cV_i, \ag_i)$ is one of the following:
\begin{itemize}
  \item
  $V_i$ is of \emph{common type}. By Proposition~\ref{p:irst}, $\dim \ag_i = 8$ and $V_i$ is as in Theorem~\ref{th:class} (\ref{it:th6}, \ref{it:th7}) (and in Lemma~\ref{l:sp2R8}).

  \item
  $V_i$ is of \emph{centralizer type}. By Corollary~\ref{c:GOct}, we obtain that either $\dim V_i = 1$ and $\dim \ag_i = 2$, or $\dim V_i \in \{2,3\}, \, \dim \ag_i = 4$, and $V_i$ lies in one of the $\so(3)$ ideals of $\so(\ag_i)$.

  \item
  $V_i$ is of \emph{Rep type}. From Section~\ref{ss:gorep}, $\dim V_i = 3, \, \dim \ag_i = 4m_i$, and $V_i$ is the image of an irreducible representation of $\so(3)$ on $\ag_i$. Note that if $m_i=1$, the subspace $V_i$ is also of centralizer type.

  \item
  $V_i$ is of \emph{Clifford type}, and is not of any of the types above. By Theorem~\ref{tha:Cliff}, $\dim \ag_i = 8$ and $\dim V_i \in \{5, 6, 7\}$.
\end{itemize}

Furthermore, as $\cV$ is non-singular, the projection $\pi_i: \cV \to \cV_i$ must be bijective, for all $i=1, \dots, p$. In particular, $\dim \cV_i = \cV$, and so $\dim\cV \in \{1,2,3,5,6,7\}$.

We consider these cases separately.

\subsubsection{$\dim \cV \in \{1,2\}$}
\label{sss:redu12}
Then $\cV$ is of centralizer type, and so by Corollary~\ref{c:GOct} we obtain the subspaces in Theorem~\ref{th:class}\eqref{it:th1} and ~\eqref{it:th2}.

\subsubsection{$\dim \cV = 3$}
\label{sss:redu3}
If $\cV$ is of centralizer type, Corollary~\ref{c:GOct} gives the subspaces in Theorem~\ref{th:class}~\eqref{it:th3}\eqref{it:th3cen}.

Suppose $\cV$ is not of centralizer type. By Lemma~\ref{l:forc}\eqref{it:cpab}, the pure centralizer $\cP(\cV)$ cannot be abelian. By Lemma~\ref{l:red}\eqref{it:redP}, it is isomorphic to a subalgebra of $\so(\cV) = \so(3)$, and so $\cP(\cV) \cong \so(3)$. Furthermore, from the above list it follows that each $\cV_i \subset \so(\ag_i)$ is of Rep type: in the decomposition $\ag = \oplus_{i=1}^p \ag_i$ we have $\dim \ag_i = 4m_i$, for $i=1, \dots, p$, and $\cV_i \subset \so(\ag_i)$ is a subalgebra which is the image of an irreducible representation $\rho_i$ of $\so(3)$ on $\ag_i$ (note that we must have $m_i > 1$ for at least one $i$, as otherwise $\cV$ is of centralizer type). For every such representation $\rho_i$ we have $\cC(\cV_i) = \spg(1) \subset \so(\ag_i)$ and $\cP(\cV_i) = \cV_i =\rho_i(\so(3)) \subset \so(\ag_i)$. It now follows from Lemma~\ref{l:red}\eqref{it:redP} that $\cP(\cV)=\rho(\so(3))$ where $\rho=\oplus_{i=1}^p \rho_i$ is the direct sum of the representations $\rho_i$. As $\cV$ is a $\cP(\cV)$-module, we obtain that there exist nonzero constants $\la_1, \dots, \la_p$ such that for a basis $E_1, E_2, E_3$ for $\so(3)$ we have $\cV = \Span_{s=1,2,3}(\sum_{i=1}^p \la_i\rho_i(E_s)) \subset \oplus_{i=1}^p \so(\ag_i) \subset \so(\ag)$.

To show that we obtain the subspaces $\cV$ which are given in Theorem~\ref{th:class}\eqref{it:th3}\eqref{it:th3rep} (we use an equivalent description given in Section~\ref{ss:dim3}), we need to show that the GO conditions for such subspaces are satisfied if and only if $\la_i = \la_j$ whenever $m_i, m_j >1$. The ``if" part is proved in Section~\ref{ss:dim3}. To the prove the ``only if" part, we can assume without loss of generality that $p=2$ and $m_1, m_2 > 1$ (by Lemma~\ref{l:redst}\eqref{it:redV}). Then we have $\cV = \Span_{s=1,2,3}(\la_1\rho_1(E_s) \oplus \la_2\rho_2(E_s))$ and $\cP(\cV) = \Span_{s=1,2,3}(\rho_1(E_s) \oplus \rho_2(E_s))$. If $m_1 \ne m_2$, the representations $\rho_1$ and $\rho_2$ are not isomorphic, and so by Lemma~\ref{l:red}\eqref{it:redCdec},\eqref{it:redCij} we find that $\cC(\cV) = \cC(\cV_1) \oplus \cC(\cV_2)$, and so $\cC(\cV)$ is isomorphic to the direct sum of the two copies $\spg(1)_i$ of $\spg(1)$, where for $i=1,2$, the subalgebra $\spg(1)_i \subset \so(\ag_i)$ is the  centralizer of $\rho_i(\so(3))$ in $\so(\ag_i)$ and is spanned by anticommuting complex structures $K_i^1, K_i^2, K_i^3 \in \so(\ag_i)$. Let $A_i = \rho_i(E), \, i=1,2$, for a nonzero $E \in \so(3)$, and denote $J = \la_1 A_1 + \la_2 A_2 \in \cV$. Then $\cC(J) = \Span(A_1 + A_2, \spg(1)_1 \oplus \spg(1)_2)$. Take generic elements $X_i \in \ag_i, \, i=1,2$. Then for $X = X_1+ X_2 \in \ag$, equation~\eqref{eq:NXJX} gives $\la_i A_i X_i = \mu A_i X_i + \sum_{s=1}^{3} \eta_i^s K_i^s X_i, \, i=1,2$, for some $\mu, \eta_i^s \in \br$. Hence $(\la_i - \mu) A_i X_i = \sum_{s=1}^{3} \eta_i^s K_i^s X_i$, and so by Lemma~\ref{l:linind} below we obtain $\la_i = \mu$, for $i=1,2$, hence $\la_1 = \la_2$, as required. A little more is needed when $m_1 = m_2 \,( =m)$. Then the representations $\rho_1$ and $\rho_2$ are isomorphic, and choosing the bases for $\ag_1$ and $\ag_2$ accordingly, we obtain by Lemma~\ref{l:red}\eqref{it:redCij} that $\cC(\cV)= \br \left(\begin{smallmatrix}0&I_{4m}\\-I_{4m}&0 \end{smallmatrix}\right) \oplus \Span_{s=1,2,3} (\left(\begin{smallmatrix} 0&K^s\\K^s&0 \end{smallmatrix}\right), \left(\begin{smallmatrix} K^s&0\\0&0\end{smallmatrix}\right), \left(\begin{smallmatrix}0&0\\0&K^s \end{smallmatrix} \right))$ (so that $\cC(\cV) \cong \spg(2)$), where the matrices $K^1, K^2, K^3$ span the subalgebra $\spg(1)$ which centralises $\rho_1(\so(3))$ in $\so(4m)=\so(\ag_1)$. Similar to the above, let $J = \left(\begin{smallmatrix}\la_1 A&0\\0& \la_2 A\end{smallmatrix}\right) \in \cV$, where $A \in \rho_1(\so(3)), \, A \ne 0$. Then $\cC(J) = \br \left(\begin{smallmatrix} A & 0\\ 0 & A\end{smallmatrix}\right) \oplus \cC(\cV)$. Taking $X=(Y,Y)^t$ in equation~\eqref{eq:NXJX}, where $Y \in \br^{4m}$ is generic, we obtain $\la_i A Y = \mu AY + \sum_{s=1}^{3} \eta_i^s K_i Y + \xi_i Y$ for $i=1,2$, where $\eta_i^s, \xi_i \in \br$. But then $(\la_i-\mu)AY \in \Span (K_1Y, K_2Y, K_3Y, Y)$, and so by Lemma~\ref{l:linind} below we obtain $\la_i = \mu$, for $i=1,2$, which again implies $\la_1 = \la_2$.

\begin{lemma} \label{l:linind}
  Suppose $A \in \so(4m), \, m > 1$, is nonzero and commutes with pairwise anticommuting complex structures $K_1, K_2, K_3 \in  \so(4m)$. Then for a generic $X \in \br^{4m}$, the vectors $AX, K_1X, K_2X, K_3X$ and $X$ are linearly independent.
\end{lemma}
\begin{proof}
  As $AX, K_1X,K_2X,K_3X \perp X$, it suffices to show that for at least one $X \in \br^{4m}$, the vectors $AX, K_1X, K_2X, K_3X$ are linearly independent. Furthermore, it is sufficient to establish this fact for $m=2$. Identifying $\br^{8}$ with the module $\bh^2$, in the notation of Section~\ref{ss:dima8}, the subalgebras $\spg(1)$ and $\spg(2)$ have the form as given in~\eqref{eq:sp2modules1}, and so we need to show that for quaternions $a, b \in \Im \bh, \, p \in \bh$, not all three being zero, there exist $x, y \in \bh$ such that for no $u \in \Im \bh$, the two equations $ax+py=xu,\, -\overline{p}x + by = yu$ are consistent. Now if $p \ne 0$, taking $x=p, \, y=1$ in the second equation we obtain $|p|^2=b-u$, which is a contradiction, as the $b-u$ is imaginary. If $p=0$, take $y=1$. Then from the second equation, $u=b$, and then the first equation gives $ax=xb$. Now if $a \ne b$, we take $x=1$, and if $a = b \,( \ne 0)$, we take a nonzero $x \perp 1, a$.
\end{proof}

\subsubsection{$\dim \cV \in \{5,6,7\}$}
\label{sss:redu567}

Let $\ag=\oplus_{i=1}^p \ag_i$ be the decomposition into $\cV$-irreducible subspaces. Then from the list above, $\dim \ag_i = 8$, for all $i=1, \dots, p$. Moreover, if $\dim \cV = 5$, then every subspace $\cV_i \subset \so(\ag_i)$ is of Clifford type, and if $\dim \cV \in \{6,7\}$, then every $\cV_i \subset \so(\ag_i)$ is either of Clifford type, or is as in Theorem~\ref{th:class} (\ref{it:th6}, \ref{it:th7}).

The proof is essentially completed by the following lemma.

\begin{lemma} \label{l:567C}
  Suppose $\dim \cV \in \{5,6,7\}$, and let $\ag=\ag_1 \oplus \ag_2$ be the orthogonal decomposition into $\cV$-irreducible subspaces. If $\cV$ is a GO subspace, then
  \begin{enumerate}[label=\emph{(\alph*)},ref=\alph*]
    \item \label{it:567bothC}
    both $\cV_1$ and $\cV_2$ are of Clifford type, and
    \item \label{it:567allC}
    $\cV$ itself is also of Clifford type.
  \end{enumerate}
\end{lemma}
\begin{proof}
  For assertion~\eqref{it:567bothC} we can assume that $\dim \cV \in \{6,7\}$. Suppose the subspace $\cV_1 \subset \so(\ag_1)$ is one of the subspaces given in Theorem~\ref{th:class} (\ref{it:th6}, \ref{it:th7}).

  We first assume that $\cC_{12} \ne 0$. Then by Lemma~\ref{l:red}\eqref{it:redCij}, relative to some orthonormal bases for $\ag_1$ and $\ag_2$, the matrices of $\pi_1(J)$ and $\pi_2(J)$ are the same, for all $J \in \cV$, and so in particular, both $\cV_1$ and $\cV_2$ are the subspaces of $\so(8)$ as given in Theorem~\ref{th:class} (\ref{it:th6}, \ref{it:th7}). By Lemma~\ref{l:sp2R8}\eqref{it:sp2R8normalizer} we have $\cC(\cV_1)=0$. Then by Lemma~\ref{l:red}(\ref{it:redCdec},\ref{it:redCij}) we find that $\cC(\cV)= \br \, \left(\begin{smallmatrix} 0 & I_8\\ -I_8 & 0\end{smallmatrix}\right)$, and by Lemma~\ref{l:red}\eqref{it:redP}, relative to the chosen bases, $\cP(\cV)=\Span(\left(\begin{smallmatrix} T&0\\0&T \end{smallmatrix}\right) \, | \, T \in \cP(\cV_1))$. Moreover, we know from Lemma~\ref{l:sp2R8}\eqref{it:sp2R8normalizer} that $\cP(\cV_1)=\spg(2)$, and that for a generic $J = \left(\begin{smallmatrix} K&0\\0&K \end{smallmatrix}\right) \in \cV$ and a generic $X_1 \in \ag_1$, there exists a unique $T \in  \cP(\cV_1)$ which satisfies the GO conditions~\eqref{eq:NJ} and~\eqref{eq:NXJX}, that is, for which $[T, K]=0$ and $TX_1=KX_1$. We now take such $J = \left(\begin{smallmatrix} K&0\\0&K \end{smallmatrix}\right) \in \cV$ and take $X=X_1+X_2 \in \ag$, with such $X_1 \in \ag_1$ and with $X_2 \in \ag_2$ in~\eqref{eq:NJ} and~\eqref{eq:NXJX}. Then there must exist $N = \left(\begin{smallmatrix} T'&0\\0&T' \end{smallmatrix}\right) + \mu \left(\begin{smallmatrix} 0 & I_8\\ -I_8 & 0\end{smallmatrix}\right) (\, \in \cP(\cV) \oplus \cC(\cV) = \cN(\cV)),\, \mu \in \br$, such that $[A, T']=0$ and $AX_1 = T'X_1 + \mu X_2, \, AX_2 = T'X_2 - \mu X_1$. Choosing $X_2 \not \perp X_1$ and taking the inner product of both sides of the equation $AX_1 = TX_1 + \mu X_2$ with $X_1$ we find that $\mu = 0$. Then we obtain $AX_1 = T'X_1 , \, AX_2 = T'X_2$, and by uniqueness, we get $T'=T$, which implies $AX_2 = TX_2$, for almost all $X_2 \in \ag_2$. Then $T=A$. As $J$ is a generic element of $\cV$ it follows that $\cV_1 \subset \cN(\cV_1)$, which is a contradiction, as $\cV_1$ is not of Rep type.

  Now assume that $\cC_{12} = 0$. Then $\cC(\cV_1)=0$ by Lemma~\ref{l:sp2R8}\eqref{it:sp2R8normalizer}. Moreover, $\cV_2$ is either of Clifford type or is as given in Theorem~\ref{th:class} (\ref{it:th6}, \ref{it:th7}), and in both cases, $\cC(\cV_2)=0$ which implies that $\cC(\cV)=0$, and so $\cN(\cV)=\cP(\cV) \subset \cP(\cV_1) \oplus \cP(\cV_2)$. We now take a generic $J = \left(\begin{smallmatrix} K_1&0\\0&K_2 \end{smallmatrix}\right) \in \cV$ and a generic $X=X_1+X_2 \in \ag$, where $X_1 \in \ag_1, \, X_2 \in \ag_2$. Then from~\eqref{eq:NJ} and~\eqref{eq:NXJX}, there must exist $N= \left(\begin{smallmatrix} N_1&0\\0&N_2 \end{smallmatrix}\right) \in \cN(\cV)$ such that $[N_i,J_i]=0$ and $J_iX_i=N_iX_i$, for $i=1,2$. By Lemma~\ref{l:red}\eqref{it:redP}, we have $N_i \in \cN(\cV_i)$, and so by Lemma~\ref{l:sp2R8}\eqref{it:sp2R8normalizer} applied to $J_1$ and $X_1$, there exists a unique such $N_1$, and hence a unique element $N \in \cN(\cV)$ such that $\pi_1(N)=N_1$ (again, by Lemma~\ref{l:red}\eqref{it:redP}). But then $N_2=\pi_2(N)$ is also uniquely determined, independently of $X_2 \in \ag_2$, and so we obtain $J_2X_2=N_2X_2$, for all $X_2 \in \ag_2$, which implies $J_2 = N_2$, a contradiction, as $\cV_2$ is not of Rep type.

  For assertion~\eqref{it:567allC}, denote $d=\dim \cV \in \{5,6,7\}$. We have $\dim \ag_1 = \dim \ag_2 = 8$. Choose orthonormal bases for $\ag_1$ and for $\ag_2$ relative to which the subspaces $\cV_1$ and $\cV_2$ are spanned by the same $8 \times 8$ anticommuting complex structures $J_k, \, k=1, \dots, d$ (which we can extend to a collection of $7$ such matrices $J_k, \, k=1, \dots, 7$, if $d < 7$). In a Euclidean space $\br^d$, choose an orthonormal basis $\{e_1, \dots, e_d\}$. For $u \in \br^d$ we denote $J_u=\sum_{k=1}^{d} u_k J_k$, where $u=\sum_{k=1}^{d} u_k e_k$. Then for some $\phi \in \End(\br^d)$ we have $\cV = \Span (\left(\begin{smallmatrix} J_u & 0\\ 0 & J_{\phi u}\end{smallmatrix}\right) \, | \, u \in \br^d)$.  Orientation-preserving orthogonal changes of the basis $\{J_1, \dots, J_7\}$ for $\Span(J_1, \dots, J_7)$ can be achieved by changing the orthonormal bases for $\ag_1$ and $\ag_2$, and so by the polar decomposition of the matrix of the endomorphism $\phi$, we can assume that $\phi = \diag(\la_1, \dots, \la_d)$, where $\la_k > 0, \, k=1, \dots, d$ when $d \le 6$, and $\la_k > 0, \, k=1, \dots, 6,\; \la_7 \ne 0$, when $d=7$. Note that $\cV$ is of Clifford type if and only if $\phi$ is orthogonal. Seeking a contradiction we now assume that it is not (which with our choice of bases, means that $\la_k^2 \ne 1$ for at least one $k=1, \dots, d$).

  We want to compute $\cN(\cV)$. By Lemma~\ref{l:red}\eqref{it:redCij} we have $\cC_{12}=0$. Furthermore, $\cC(\cV_1)=\cC(\cV_2)=0$ if $d=6,7$, and $\cC(\cV_1)=\cC(\cV_2)=\br J'$, where $J'=J_6J_7$, if $d =5$, and so by Lemma~\ref{l:red}\eqref{it:redCdec} we obtain that $\cC(\cV)=0$ when $d=6,7$, and $\cC(\cV) = \Span(\left(\begin{smallmatrix} J' & 0\\ 0 & 0\end{smallmatrix}\right),\left(\begin{smallmatrix} 0 & 0\\ 0 & J'\end{smallmatrix}\right))$ when $d=5$. The pure normalizer of $\Span(J_1, \dots, J_d)$ in $\so(8)$ is isomorphic to $\so(d)$ and is given by $\Span_{1 \le k < l \le d}(J_kJ_l)$. For a matrix $B \in \so(d)$, denote $F_B = \sum_{k,l} B_{kl} J_kJ_l \in \so(8)$. Then for $B,C \in \so(d)$ we have $\left(\begin{smallmatrix} F_B& 0\\ 0 & F_C \end{smallmatrix}\right) \in \cP(\cV)$ if and only if $\phi B = C \phi$. This implies that $B$ commutes with $\phi^t \phi$. In particular, if $d < 7$, then $B$ commutes with $\phi$ (as the matrix of $\phi$ is diagonal, with positive diagonal entries), and hence $B=C$.

  We now consider the cases $d=5,6,7$ separately and show that there is a contradiction with the GO condition, in all three cases.

  Suppose $d = 5$. Then $\cP(\cV)$ is spanned by the matrices $\left(\begin{smallmatrix} F_B & 0\\ 0 & F_B\end{smallmatrix}\right)$, where $B$ takes values in some subalgebra of $\so(5)$, and  $\cC(\cV) = \Span(\left(\begin{smallmatrix} J' & 0\\ 0 & 0\end{smallmatrix}\right),\left(\begin{smallmatrix} 0 & 0\\ 0 & J'\end{smallmatrix}\right))$. Suppose $\la_1 \ne 1$. In the equation~\eqref{eq:NXJX}, take $J =\left(\begin{smallmatrix} J_1 & 0\\ 0 & \la_1 J_1\end{smallmatrix}\right) \in \cV$ and $X=(Y,Y)^t \in \ag$, where $Y \in \br^8$ is such that $J_1Y$ and $J'Y$ are linearly independent. We obtain $J_1Y=F_B Y+ \alpha J' Y$ and $\la_1 J_1Y=F_BY+ \beta J' Y$, for some $\alpha, \beta \in \br$ and $B \in \so(5)$, which implies $(1-\la_1)J_1Y=(\alpha-\beta)J'Y$, a contradiction.

  If $d = 6$, the argument is even easier, as $\cC(\cV) = 0$. Take $J =\left(\begin{smallmatrix} J_1 & 0\\ 0 & \la_1 J_1\end{smallmatrix}\right) \in \cV$ (where $\la_1 \ne 1$) and $X=(Y,Y)^t \in \ag$, where $Y \in \br^8$ is such that $J_1Y$ and $J'Y$ are linearly independent. Then from~\eqref{eq:NXJX} we obtain $J_1Y=\la_1 J_1Y=F_B Y$, a contradiction.

  Now let $d = 7$. If $\la_7 > 0$, we repeat the above argument again. Suppose $\la_7 < 0$. Then for $\left(\begin{smallmatrix} F_B& 0\\ 0 & F_C \end{smallmatrix}\right) \in \cP(\cV)$ we have $\phi B = C \phi$ (and so $B$ commutes with $\phi^t \phi$, but not necessarily with $\phi$ itself). Take $J =\left(\begin{smallmatrix} J_7 & 0\\ 0 & \la_7 J_7\end{smallmatrix}\right) \in \cV$. Then $\cC(J)$ is spanned by the matrices $\left(\begin{smallmatrix} F_B& 0\\ 0 & F_C \end{smallmatrix}\right)$, where $B, C \in \so(7)$ are such that $Be_7=Ce_7=0$ and $\phi B = C \phi$. Then we still get $B=C$, and choosing any nonzero $X=(Y,Y)^t \in \ag$ in~\eqref{eq:NXJX} we obtain  $J_7Y=\la_7 J_7Y=F_B Y$, again a contradiction.
\end{proof}

From Lemma~\ref{l:567C} it follows that $\ag$ can be reducible only when $\cV$ is of Clifford type.
This completes the proof of Theorem~\ref{th:class}.

\end{document}